\documentclass[final,leqno]{siamltex}
\title{On the alternating randomized block Kaczmarz method}
\author{Nian-Ci Wu\footnotemark[1]\ \footnotemark[4] ,
        Yang Zhou\footnotemark[2] , \and Zhaolu Tian\footnotemark[3]}

\begin{document}
\maketitle
\renewcommand{\thefootnote}{\fnsymbol{footnote}}

\footnotetext[1]{School of Mathematics and Statistics, South-Central Minzu University, Wuhan 430074, China.}
\footnotetext[2]{School of Mathematics and Computational Science, Huaihua University, Huaihua 418000, China.}
\footnotetext[3]{College of Applied Mathematics, Shanxi University of Finance and Economics, Taiyuan 030006, China.}
\footnotetext[4]{Corresponding author. E-mail address: {\sf  nianciwu@scuec.edu.cn} (N.-C. Wu).}
\renewcommand{\thefootnote}{\arabic{footnote}}

\begin{abstract}
The block Kaczmarz method and its variants are designed for solving the over-determined linear system. They involves iteratively projecting the current point onto the solution space of a subset of constraints. In this work, by alternately dealing with two subproblems (i.e., linear system with multiple right-hand sides) using the block Kaczmarz method, we propose the {\it Alternating Randomized Block Kaczmarz} (ARBK) method to solve the linear matrix equation $AXB=F$, which incorporates a randomized index selection scheme to determine the subset of constraints. The convergence analysis reveals that the ARBK method has a linear convergence rate bounded by an explicit expression. Several numerical studies have been conducted to validate the theoretical findings.
 \end{abstract}



\pagestyle{myheadings}
\thispagestyle{plain}
\markboth{The alternating randomized block Kaczmarz method}{N.-C. Wu, Y. Zhou, and Z. Tian}

\section{Introduction}
Geometric iterative method (GIM) is widely used in large-scale data fitting because of its intuitive geometric meaning and efficiency; see \cite{16CHL,14DL,02Farin, 22JWH, 02LPRM,18LMD,97PT} and the references therein. For the case of surface fitting \cite{20Liu}, it often leads to a large sparse system of linear matrix equations formed by
\begin{equation}\label{AXB=F}
AXB=F.
\end{equation}
Here, $A \in \mathbb{R}^{m\times n}$ and $B \in \mathbb{R}^{p\times q}$ are the collocation matrices of the B-spline based on the parameter sequences and knot vectors, $F \in \mathbb{R}^{m\times q}$ is a data-point coordinate matrix, and $X \in \mathbb{R}^{n\times p}$ is an control coefficient matrix to be solved. There are no restrictions on the dimensions of $A$ and $B$ depending on the amount of data in a given experiment.

Iterative algorithms are often well-studied for solving the above large-scale problem, such as matrix splitting iterative method \cite{11Bai,20Liu,17Tian,13Wang} and gradient-type iterative method \cite{08Ding, 03Liao}. In the application of surface fitting with the number of data points being very large, the matrix $A$ is thin/tall (i.e., $m > n$) and $B$ is fat/flat (i.e., $p<q$), and so sparse that only a rather few number of elements can be stored. This naturally limits the scope of the matrix splitting iterative method.

Gradient-type iterative method is an optimization algorithm that the next iterate follows the negative gradient of an objective function in order to locate the minimum of the function. Let $G(X)$ be a differentiable function.  It is formulated by
\begin{align*}
  X^{(k+1)} = X^{(k)} - \alpha \nabla G(X^{(k)})
\end{align*}
for $k=0,1,2,\cdots$, where $\alpha$ is a step-size and $\nabla G(X)$ is the gradient of $G(X)$. Some well-known methods can be written in this form with appropriate choices of $G(X)$. For example,
\begin{enumerate}[$\diamond$]
\setlength{\itemindent}{0.5cm}
\addtolength{\itemsep}{-0.1em} 
\item gradient-based iterative method \cite[Theorem 2]{08Ding}:
\begin{align*}
  G(X) = \frac{1}{2}\bF{C-AXB}^2;
\end{align*}
\item Kaczmarz (also known as row-action) method \cite[Equation (4.2)]{22WLZ}:
\begin{align*}
  G (X) = \frac{1}{2}\frac{|F_{ij}-A_{i,:} X B_{:,j}|^2}{\| A_{i,:} \|^2 \| B_{:,j} \|^2},
\end{align*}
where the symbols $M_{i,:}$, $M_{:,j}$,  and $M_{i,j}$ denote the $i$th row, $j$th column, and $(i,j)$th element for matrix $M$, respectively.
\end{enumerate}

Due to the cheap per iteration arithmetic cost and low total computational complexity, the row-action method \cite{81Censor} becomes overwhelming for solving matrix equation \eqref{AXB=F} with general coefficient matrices, such as the global randomized block Kaczmarz algorithm and and its pseudoinverse-free form in \cite{22Niu};
 the sketch-and-project method for solving matrix equation in \cite{23BGLLW}, which unifies a variety of randomized iterative methods by varying the parameter matrices;
 the matrix equation relaxed greedy randomized Kaczmarz and  maximal weighted residual Kaczmarz methods in \cite{22WLZ}.
 We also refer the reader to a recent comprehensive review of solving matrix equation in \cite{16Sim} and the references therein for more details.

The {\it Alternating Direction Implicit} (ADI) method \cite{55PR} was introduced by Peaceman and Rachford specifically for solving the partial differential equations. Inspired by the idea of ADI, we consider to solve matrix equation \eqref{AXB=F} in the $X$ and $Y$ directions alternatively.
A computational framework of ADI for solving \eqref{AXB=F} is presented as follows.

\vskip 0.5ex
\begin{framed} 
\begin{algorithmic}[]
\setlength{\itemindent}{0.5cm}
\State \qquad {\bf for} $k=0,1,2,\cdots$ {\bf do}
\State \qquad \quad solve $AY = F$ and obtain $Y^{(k+1)}$
\State \qquad \quad solve $XB=Y^{(k+1)}$ and obtain $X^{(k+1)}$
\State \qquad {\bf endfor}.
\end{algorithmic}
\end{framed}
\vskip 0.5ex

Actually, the ADI process requires solving two subproblems, i.e., linear systems with multiple right-hand sides, formed by
\begin{align}\label{MatrixAXB=F}
AY=F \quad {\rm and} \quad XB=Y~~({\rm or} ~~B^TX^T = Y^T),
\end{align}
where $X$ and $Y$ are unknown. Recently, Xing et al. \cite{23XBL} utilized two signal Kaczmarz iterates \cite{37Kac} and proposed a so-called consistent matrix equation randomized Kaczmarz (CME-RK) method for \eqref{AXB=F}. From geometric aspect, we find that, in the CME-RK method, the auxiliary quantity  $ Y_{:,j}^{(k+1)}$ and the next approximation $ X_{i,:}^{(k+1)}$ are the projections of $ Y_{:,j}^{(k)}$ and $X_{i,:}^{(k)}$ onto the solution space of $A_{i,:} X = F_{i,:}$ and $X B_{:,j} = Y^{(k+1)}_{:,j}$, respectively, in which only one single equation is used. It is natural to project the current point onto the solution space of many equations simultaneously by selecting a block of rows or columns. To the best of our knowledge, the block variant of CME-RK is new and has not been previously proposed or analyzed.

In this work, we will propose a randomized block Kaczmarz-type algorithm to solve \eqref{AXB=F} and call it the {\it Alternating Randomized Block Kaczmarz} (ARBK) method. The organization of this work is as follows. We first present the ARBK method in Section \ref{AXB=F:ARBK} and then establish its convergence theory in Section \ref{AXB=F:ARBK_Conver}. 
Numerical experiments are illustrated in Section \ref{AXB=F:Numer} to show the effectiveness of our method. Finally, we end this paper with some conclusions in Section \ref{AXB=F:Conclusion}.

Some notation and definitions used in this paper are as follows.  For any $M\in \mathbb{R}^{m\times n}$, we use $M^{T}$, $M^{\dag}$, $|| M ||_{2}$, $|| M ||_{F}$, ${\rm Tr}(M)$, $\sigma_{\max}(M)$, and $\sigma_{\min}(M)$ to denote the transpose, the Moore-Penrose pseudoinverse, the 2-norm, the Frobenius norm, the trace, the maximum singular value, and the minimum singular value, respectively. In addition, we denote by $M_{U,:}$ and $M_{:,V}$ the submatrices of $M$,  whose rows and columns are indexed by the index sets $U$ and $V$, respectively. For any two matrices $M_1$ and $M_2$ with same sizes, we define the Frobenius inner product $\langle M_1,M_2 \rangle_{F} = {\rm Tr}(M_1^{T}M_2) = {\rm Tr}(M_1M_2^{T})$. Specially, $\langle M_1,M_1\rangle_{F}=|| M_1 ||_F^2$. The identity matrix of size $n$ is given by $I_n$. For any integer $n$, we use $[n]$ to express the set $\left\{1,2,\cdots, n \right\}$. In the end the symbol $|U|$ denotes the cardinality of a set $U$.

\section{The ARBK method}\label{AXB=F:ARBK} In this section, we first briefly review the CME-RK method in \cite{23XBL}, then develop the ARBK method to solve matrix equation $AXB=F$ by transforming the random index selection into matrix form, and finally  discuss the randomized index selection scheme determining the subset of constraints.

\subsection{The CME-RK method} The CME-RK  method includes two basic computational procedures. The first-half step applies a signal Kaczmarz iterate to the system $AY=F$ and outputs  $Y^{(k+1)}$, and the second-half step utilizes another signal Kaczmarz iterate on the system $XB =Y^{(k+1)}$ and outputs $X^{(k+1)}$. Continuing with this process alternatively, the CME-RK   method in \cite{23XBL} is formulated as Algorithm \ref{alg:CME-RK}.

\begin{algorithm}[!htb]
\caption{The CME-RK method (\cite[Algorithm 2.1]{23XBL})}\label{alg:CME-RK}
\begin{algorithmic}[1]
\Require
The  matrices $A \in \mathbb{R}^{m\times n}$ and $B \in \mathbb{R}^{p\times q}$ , and a right-hand side $F \in \mathbb{R}^{m\times q}$, initial matrices $X^{(0)} \in \mathbb{R}^{n\times p}$, $Y^{(0)} = X^{(0)}B$.
\State {\bf for} $k=0,1,2,\cdots,K-1$ {\bf do}
\State \quad select $i\in [m]$ at random with probability $P\left({\rm Index} =  i_k \right) =  \| A_{i_k,:}\|^2/\| A\|^2_F$;
\State \quad compute
\begin{align}\label{ARBK:CME-RK_iterate1}
Y^{(k+1)} = Y^{(k)} + \frac{1}{\| A_{i_k,:}\|^2} A_{i_k,:}^T \left(F_{i_k,:} - A_{i_k,:}  Y^{(k)} \right);
\end{align}
\State \quad select $i\in [m]$ at random with probability$P\left({\rm Index} =  j_k \right) =  \| B_{:,j_k}\|^2/\| B\|^2_F$;
\State \quad compute
\begin{align}\label{ARBK:CME-RK_iterate}
X^{(k+1)} = X^{(k)} + \frac{1}{\| B_{:,j_k}\|^2} \left(Y^{(k+1)}_{:,j_k} - X^{(k)}  B_{:,j_k}  \right)  B_{:,j_k}^T;
\end{align}
\State {\bf endfor}.
\Ensure
$X^{(K)}$.
\end{algorithmic}
\end{algorithm}

We observe that the Kaczmarz iterates in \eqref{ARBK:CME-RK_iterate1} and \eqref{ARBK:CME-RK_iterate} can be induced by using the Petrov-Galerkin  conditions (PGC) \cite{Saad2003}. The process is elaborated in the following remark.

\begin{remark}
Let $ \mathcal{L} $ and  $ \mathcal{K} $ be the constrained and search subspaces, respectively. According to PGC, we need to assume that the $j$th column of $Y^{(k+1)}$ to $AY_{:,j}=F$, i.e., $Y_{:,j}^{(k+1)}$, belongs to the affine space $Y^{(k)}_{:,j}+ \mathcal{K} $  and satisfy the constraint $F_{:,j}-AY_{:,j}^{(k+1)}$ being orthogonal to $ \mathcal{L} $ with $j\in [q]$ for $k=0,1,2,\cdots$. Namely,
\begin{align}\label{ARBK:PG_RK}
 Y^{(k+1)}_{:,j} \in Y^{(k)}_{:,j} + \mathcal{K}  ~~ and ~~  F_{:,j}-AY^{(k+1)}_{:,j} \perp \mathcal{L}.
\end{align}
In particular, assume that
\begin{equation*}
    \mathcal{K} ={\rm span}\left\{A_{i,:}^T\right\}
   ~~ and ~~
    \mathcal{L} ={\rm span}\left\{\mu_i\right\},
\end{equation*}
where $\mu_{i}$ is the $i$th column of  $I_m$ for $i\in [m]$, the constrains in \eqref{ARBK:PG_RK} can be equivalently rewritten by
\begin{align*}
 Y^{(k+1)}_{:,j} = Y^{(k)}_{:,j} +  \eta A_{i,:}^T
 ~~ and ~~
   \left\langle F_{:,j}-AY^{(k+1)}_{:,j},~\mu_i \right\rangle  =0,
\end{align*}
where $\eta$ is a constant. After an elementary calculation, we have
\begin{align*}
 Y^{(k+1)}_{:,j}
 = Y^{(k)}_{:,j} +  A_{i,:}^T \left( \mu_i^T A A_{i,:}^T \right)^\dag \mu_i^T \left(F_{:,j}-AY^{(k)}_{:,j}\right)
 = Y^{(k)}_{:,j} +  \frac{F_{i,j}-A_{i,:}^TY^{(k+1)}_{:,j}}{\|A_{i,:}\|^2}A_{i,:}^T.
\end{align*}
Then, it follows that
\begin{align*}
  Y^{(k+1)} & =
  \begin{bmatrix}
    Y^{(k+1)}_{:,1} & Y^{(k+1)}_{:,2}& \cdots & Y^{(k+1)}_{:,q}
  \end{bmatrix} \\
  & = Y^{(k)}  + \begin{bmatrix}
    \frac{F_{i,1}-A_{i,:}^TY^{(k+1)}_{:,1}}{\|A_{i,:}\|^2}A_{i,:}^T & \frac{F_{i,2}-A_{i,:}^TY^{(k+1)}_{:,2}}{\|A_{i,:}\|^2}A_{i,:}^T & \cdots &
    \frac{F_{i,q}-A_{i,:}^TY^{(k+1)}_{:,q}}{\|A_{i,:}\|^2}A_{i,:}^T
  \end{bmatrix} \\
  & = Y^{(k)} + \frac{1}{\| A_{i,:}\|^2} A_{i,:}^T \left(F_{i,:} - A_{i,:}  Y^{(k)} \right).
\end{align*}
It is a similar story for computing $X^{(k+1)}$, so we omit the derivation.
\end{remark}

The convergence result of the CME-RK method in \cite{23XBL} is presented as follows. For more details, we refer to Theorem 2.1 in this reference.

\vskip 0.5ex
\begin{theorem}
(\cite[Theorem 2.1]{23XBL})
 The sequence $\left\{X^{(k)}\right\}_{k=0}^{\infty}$ generated by the CME-RK method starting from the initial matrix
 $X^{(0)} \in \mathbb{R}^{n\times p}$ and $Y^{(0)} = X^{(0)}B$, converges linearly to the solution $X^{\ast} := A^{\dag} F B^{\dag}$ of the consistent matrix equation \eqref{AXB=F} in mean square if the $\left(X^{(0)}_{i,:}\right)^T$ and $ Y^{(0)}_{:,j}$ are the column spaces of $B$ and $A^T$ for $i\in[n]$ and $j\in[q]$, respectively.
Moreover, the following relationship holds
\begin{align*}
E\left[||X^{(k+1)} - X^\ast||_F^2\right] \leq
\left( 1 + \frac{\sigma_{\max}^2(B)}{||B||_F^2}  \sum\limits_{\ell=0}^{k} \rho_1^{\ell+1}/\rho_2^{\ell+1} \right) \rho_2^{k+1}
||X^{(0)} - X^\ast ||_F^2,
\end{align*}
where the convergence factors $\rho_1 = 1 -  \sigma_{\min}^2(A)/||A||_F^2$ and $\rho_2 = 1 -  \sigma_{\min}^2(B)/||B||_F^2$.
\end{theorem}
\vskip 0.5ex

\subsection{The proposed method}
For $i\in [m]$ and $j\in[q]$, the $i$th row of $A$ and the $j$th column of $B$  can be written by
\begin{align*}
  A_{i,:} = \mu_i^T A ~{\rm and}~B_{:,j} = B \nu_j,
\end{align*}
respectively, where $\nu_j$ is the $j$th column of $I_q$. It implies the random selection of rows or columns in the CME-RK method is equivalent to use two independent random vectors $\mu_i^T$ and $\nu_j$ to respectively multiply the left of $A$ or the right of $B$. This random strategy can be generalized by matrices to sketch the rows and columns in any desirable way.

Let the constrained subspace and  the search subspace correspond to $ \mathcal{L} ={\rm span}\{ W_1 \}$ and  $ \mathcal{K} = {\rm span}\{ Z_1 \}$, respectively, where $W_1$ and $Z_1$ are two parameter matrices. Applying PGC to the first linear system $AY=F$, the $j$th column of $Y^{(k+1)}$ is computed by
\begin{align*}
Y^{(k+1)}_{:,j}  = Y^{(k)}_{:,j} + Z_1\left(W_1^T A Z_1\right)^\dag W_1^T \left(F_{:,j}-AY^{(k)}_{:,j}\right)
\end{align*}
for $k=0,1,2,\cdots$. Equivalently,
\begin{align}\label{eq:PG-iterate1}
Y^{(k+1)}  = Y^{(k)} + Z_1\left(W_1^T A Z_1\right)^\dag W_1^T \left(F-AY^{(k)}\right).
\end{align}
After giving another two parameter matrices $W_2$ and $Z_2$, for the second linear system $XB=Y^{(k+1)}$, a similar iteration scheme is given by
\begin{align}\label{eq:PG-iterate2}
X^{(k+1)}  = X^{(k)} + \left(Y^{(k+1)} - X^{(k)}B \right)W_2^T \left(Z_2 B W_2^T\right)^\dag Z_2.
\end{align}


Geometrically, rather than project the current points ($X^{(k)},~Y^{(k)}$) onto the solution space of only one equation in the CME-RK method,  we project them onto the solution space of many equations simultaneously by selecting a block of rows or columns. It can be achieved by taking
\begin{equation*}
   \mathcal{K} ={\rm span}\left\{ A_{U_{i},:}^T  \right\}
   ~ {\rm and} ~
   \mathcal{L} ={\rm span}\left\{ I_{:,U_{i}} \right\},
\end{equation*}
where $U_{i,:}$ is a subset of row index $[m]$ and $I_{:,U_{i}}$ denotes a column concatenation of $I_m$ indexed by $U_{i}$. Under this circumstance,  we obtain a block Kaczmarz iteration as follows.
\begin{align*}
Y^{(k+1)} = Y^{(k)} + A^{\dag}_{U_{i},:}\left(F_{U_{i},:} - A_{U_{i},:} Y^{(k)}\right)
\end{align*}
for each iteration $k \geq 1$. We follow the same route and derive another block Kaczmarz iteration as
\begin{align*}
X^{(k+1)} = X^{(k)} + \left(Y_{:,V_{j}}^{(k+1)} - X^{(k)} B_{:,V_{j}}\right) B^{\dag}_{:,V_{j}},
\end{align*}
where $V_{j} \subset [q]$ is a column index subset.

These four parameter matrices $W_1$, $Z_1$, $W_2$, and $Z_2$ can also be viewed as the preconditioner to accelerate the convergence or a selector of the particular solution of the stochastic optimization problem, which will  affect the solution set of the best approximation problem and its dual; for example, see \cite{20LR}. We emphasize that the above list is by no means comprehensive and merely serves the purpose to illustrate the elasticity  of formulas \eqref{eq:PG-iterate1} and \eqref{eq:PG-iterate2}.

The above yields the ARBK computational framework for solving matrix equation \eqref{AXB=F} which is summarized in Algorithm \ref{alg:1}.

\vskip 0.5ex
\begin{algorithm}[!htbp]
\caption{The $\mathrm{ARBK}$ method for $AXB=F$}\label{alg:1}
\begin{algorithmic}[1]
\Require
The  matrices $A \in \mathbb{R}^{m\times n}$ and $B \in \mathbb{R}^{p\times q}$ , and a right-hand side $F \in \mathbb{R}^{m\times q}$, initial matrices $X^{(0)} \in \mathbb{R}^{n\times p}$, $Y^{(0)} = X^{(0)}B$, row index set $S=\left\{U_0,U_1,\cdots,U_{s-1}\right\}$, column index set $T=\left\{V_0,V_1,\cdots,V_{t-1}\right\}$.
\State {\bf for} $k=0,1,2,\cdots,K-1$ {\bf do}
\State \quad select $U_{i_k}\in S$ uniformly at random;
\State \quad compute
\begin{align*}
Y^{(k+1)} = Y^{(k)} + A^{\dag}_{U_{i_k},:}\left (F_{U_{i_k},:} - A_{U_{i_k},:} Y^{(k)}\right);
 \end{align*}
\State \quad select $V_{j_k}\in T$ uniformly at random;
\State \quad compute
\begin{align*}
X^{(k+1)} = X^{(k)} + \left(Y^{(k+1)}_{:,V_{j_k}} - X^{(k)} B_{:,V_{j_k}}\right) B^{\dag}_{:,V_{j_k}};
\end{align*}
\State {\bf endfor}.
\Ensure
$X^{(K)}$.
\end{algorithmic}
\end{algorithm}

\vskip 0.5ex
\begin{remark}
If one chooses the row and column index sets to be $U_{i_k} = \left\{ i_k \right\}$ and $V_{j_k} = \left\{ j_k \right\}$ with $i_k\in [m]$ and $j_k \in [q]$, the iteration schemes of computing $Y^{(k+1)}$ and $X^{(k+1)}$ in the ARBK method automatically reduce to formulas \eqref{ARBK:CME-RK_iterate1} and \eqref{ARBK:CME-RK_iterate}, respectively. It comes from the fact that $A^{\dag}_{U_{i_k},:} =  A_{i_k,:}^T/\| A_{i_k,:}\|^2$ and $B^{\dag}_{:,V_{j_k}} = B_{:,j_k}^T/\| B_{:,j_k}\|^2$.
\end{remark}

\vskip 0.5ex
\begin{remark}\label{ARBK:remark_CMERKcon}
  In practice, there are infinitely many ways to construct a well-defined probability criterion to select $(U_{i_k},V_{j_k})$. For example, the index sets in the CME-RK method \cite{23XBL} are selected according to
  \begin{align*}
    P\left(U_{i_k} =  i_k \right) =  \frac{\| A_{i_k,:}\|^2}{\| A\|^2_F}~~and~~
    P\left(V_{j_k} =  j_k \right) =  \frac{\| B_{:,j_k}\|^2}{\| B\|^2_F}
  \end{align*}
  for $i_k \in [m]$ and  $j_k \in [p]$. In addition, the index sets in the GRBK method \cite{22Niu} are selected according to the non-uniformly probability distributions, i.e.,
  \begin{align*}
    P\left(U_{i_k} \right) =  \frac{\| A_{U_{i_k},:}\|^2_F}{\| A\|^2_F}~~and~~
    P\left(V_{j_k} \right) =  \frac{\| B_{:,V_{j_k}}\|^2_F}{\| B\|^2_F}
  \end{align*}
  for $U_{i_k} \subset [m]$ and  $V_{j_k} \subset [p]$.
\end{remark}

\vskip 0.5ex
\begin{remark}
The conditioning of the blocks $A_{U_{i_k},:}$ and $B_{:,V_{j_k}}$ in the ARBK method may play an important role in determining its convergence. If each block is well-conditioned, its pseudoinverse can be efficiently computed by using an inexact inner iterative method such as the conjugate gradient method (see, e.g., \cite{Saad2003}) for the least-squares problem. If not and to avoid calculating the pseudoinverse directly, very  recently, Wu and Chang proposed a semi-randomized block Kaczmarz method with simple random sampling for solving \eqref{MatrixAXB=F} \cite{22WC}, where the index sets are constructed by utilizing the partially greedy selection rule. Note that it will be an interesting work to extend this technique to solve matrix equation \eqref{AXB=F} incorporated with the idea of ADI.
\end{remark}

\subsection{How to construct the index sets}  To ensure the desirable properties, such as the well-condition of the sub-matrices and the index selection being easy to implement, Needell and Tropp employed a index {\it paving} strategy for a matrix as described below \cite{14NT}.
\vskip 1.5ex
\begin{definition}\label{ARBK:definition_paving}
(\cite[Definition 1.1]{14NT}) An $(s, \alpha_A, \beta_A)$ row paving of a matrix $A$ is a partition  of the row indices $S=\left\{U_0,U_1,\cdots,U_{s-1}\right\}$ that verifies
\begin{align*}
  \alpha_A \leq \lambda_{\min}\left(A_{U,:}A_{U,:}^T\right)
  ~~  and  ~~
  \lambda_{\max}\left(A_{U,:}A_{U,:}^T\right) \leq \beta_A
\end{align*}
for each $U \in S$. The number $s$ of blocks is called the size of the paving. The numbers $\alpha_A$ and $\beta_A$ are called lower and upper paving bounds. The ratio $\beta_A/\alpha_A$ gives a uniform bound on the squared condition number $\kappa^2(A_{U,:})$ for each subset $U$. 
\end{definition}

\vskip 1.5ex
In a similar manner, we can define the column paving for matrix $B$.
\vskip 1.5ex

\begin{definition}\label{ARBK:definition_pavingB}
 An $(t, \alpha_B, \beta_B)$ column paving of a matrix $B$ is a partition  of the column indices $T=\left\{V_0,V_1,\cdots,V_{t-1}\right\}$ that verifies
\begin{align*}
  \alpha_B \leq \lambda_{\min}\left(B_{:,V}^T B_{:,V}\right)
  ~~ and ~~
  \lambda_{\max}\left(B_{:,V}^T B_{:,V}\right) \leq \beta_B
\end{align*}
for each $V \in T$. The number $t$ of blocks is called the size of the paving. The numbers $\alpha_B$ and $\beta_B$ are called lower and upper paving bounds. The ratio $\beta_B/\alpha_B$ gives a uniform bound on the squared condition number $\kappa^2(B_{:,V})$ for each subset $V$.
\end{definition}

Based on data random sampling, several results appeared in randomized algorithms; see, for instance, references \cite{23DGR,14NT,15NZZ,08WLRT}. The sampled subset can be drawn from some specific distributions, such as subsampled random Fourier transform \cite{08WLRT}, subsampled random Hadamard transform \cite{23DGR}, and fast incoherence transform (FIT) \cite{14NT}. Specially, the FIT random matrix is defined by $\widehat{S} = \widehat{F} \widehat{E}$, where $\widehat{F}$ is an unitary discrete Fourier transform matrix and $\widehat{E}$ is a diagonal matrix whose entries are independent Rademacher random variables. This FIT random matrix can make sure that the rows or column of coefficient matrices exhibit a small amount of diversity. For more details, we refer to Proposition 3.7 in \cite{14NT}.

Closely related to the data random sampling is the random partition, which is a popular representative among pavings due to its simplicity and efficiency \cite{14NT}. The idea is to divide the row or column index of the matrix into random blocks with approximately equal size. The specific implementation is described as follows.

\vskip 0.5ex
\begin{definition}\label{ARBK:definition-RP}
 \cite[Definition 3.2]{14NT}. Suppose that $\pi_U$ is a permutation on $\{1, 2, \cdots, m\}$, chosen uniformly at random. For each $i = 1,2,\cdots,s$, define the set
 \begin{align*}
   U_{i-1} = \left\{
   \pi_U(k):~k =
   \lfloor (i-1)m/s\rfloor+1,
   \lfloor (i-1)m/s\rfloor+2,\cdots,
   \lfloor im/s\rfloor
   \right\}.
 \end{align*}
We call $S=\left\{U_0,U_1,\cdots,U_{s-1}\right\}$ as a row random partition of $[m]$ with $s$ blocks.
\end{definition}

\vskip 0.5ex
Similarly, we define another random partition of column index as follows.

\vskip 0.5ex
\begin{definition}\label{ARBK:definition-RP2}
Suppose that $\pi_V$ is a permutation on $\{1, 2, \cdots, q\}$, chosen uniformly at random. For each $j = 1,2,\cdots,t$, define the set
 \begin{align*}
   V_{j-1} = \left\{
   \pi_V(k):~k =
   \lfloor (j-1)q/t\rfloor+1,
   \lfloor (j-1)q/t\rfloor+2,\cdots,
   \lfloor jq/t\rfloor
   \right\}.
 \end{align*}
We call $T=\left\{V_0,V_1,\cdots,V_{t-1}\right\}$ as a column random partition of $[q]$ with $t$ blocks.
\end{definition}

%
%

\section{Convergence analysis of the ARBK method}\label{AXB=F:ARBK_Conver} The aim of this section is to demonstrate the convergence property of the ARBK method. First we present two lemmas which we will use in our convergence proof.

\begin{lemma}\label{AY-norma1}
{\rm (\cite[Lemma 1]{22DRS})}~Let $A\in \mathbb{R}^{m\times n}$ and $B\in \mathbb{R}^{p\times q}$ be given. Let
\begin{align*}
  \mathcal{R}(A, B):=\left\{ M \in \mathbb{R}^{n\times p}~|~\exists~ Z \in \mathbb{R}^{m\times q}~s.t. ~M = A^TZB^T \right\}.
\end{align*}
For any matrix $M\in \mathcal{R}(A, B)$, it holds that
\begin{align*}
||AMB||_F^2 \geq \sigma_{\min}^2(A)\sigma_{\min}^2(B)||M||_F^2.
\end{align*}
\end{lemma}

\begin{lemma}\label{AY-norma2}
Let the first matrix equation in \eqref{MatrixAXB=F} be consistent and $S$ be the row random partition of $[m]$ defined by Definition \ref{ARBK:definition-RP}.  Starting from the initial matrix $Y^{(0)}\in \mathcal{R}(A, I_{p\times q})$, the sequence $\left\{Y^{(k)}\right\}_{k=0}^{\infty}$ generated by Algorithm \ref{alg:1} converges linearly to the minimum Frobenius-norm solution $Y^\ast = A^{\dag} F$ in mean square form. Moreover, the solution error in expectation for the iteration sequence obeys
\begin{equation}\label{E_Yk}
E\left[||Y^{(k)} - Y^\ast||_F^2\right] \leq \widehat{\gamma}^{k} ||Y^{(0)} - Y^\ast||_F^2,
\end{equation}
where the convergence factor $\widehat{\gamma} = 1 - \sigma_{\min}^2(A)/(s\beta_A)$.
\end{lemma}

\begin{proof}
According to the update format of $Y^{(k+1)}$, we have
\begin{align*}
Y^{(k+1)} -Y^\ast &= Y^{(k)} - Y^\ast + A^{\dag}_{U_{i_k},:}(F_{U_{i_k},:} - A_{U_{i_k},:} Y^{(k)})\\
                  &= Y^{(k)} - Y^\ast + A^{\dag}_{U_{i_k},:}(A_{U_{i_k},:} Y^{\ast} - A_{U_{i_k},:} Y^{(k)})\\
                  &= (I_n - A^{\dag}_{U_{i_k},:}A_{U_{i_k},:})(Y^{(k)} - Y^\ast ).
\end{align*}
Based on the properties of the Moore-Penrose pseudoinverse, it follows that
\begin{align*}
||Y^{(k+1)} -Y^\ast ||_F^2 &= {\rm Tr}\left[(Y^{(k+1)} -Y^\ast )^T(Y^{(k+1)} -Y^\ast )\right]\\
                           &= {\rm Tr}\left[(Y^{(k)} - Y^\ast )^T(I_n - A^{\dag}_{U_{i_k},:}A_{U_{i_k},:})^T(I_n - A^{\dag}_{U_{i_k},:}A_{U_{i_k},:})(Y^{(k)} - Y^\ast )\right]\\
                           &= {\rm Tr}\left[(Y^{(k)} - Y^\ast )^T(I_n - A^{\dag}_{U_{i_k},:}A_{U_{i_k},:})(Y^{(k)} - Y^\ast )\right]\\
                           &= {\rm Tr}\left[(Y^{(k)} - Y^\ast )^T(Y^{(k)} - Y^\ast )\right] - {\rm Tr}\left[(Y^{(k)} - Y^\ast )^T A^{\dag}_{U_{i_k},:}A_{U_{i_k},:} (Y^{(k)} - Y^\ast )\right]\\
                           &= ||Y^{(k)} - Y^\ast ||_F^2 - {\rm Tr}\left[(Y^{(k)} - Y^\ast )^T A^{\dag}_{U_{i_k},:}A_{U_{i_k},:} A^{\dag}_{U_{i_k},:}A_{U_{i_k},:} (Y^{(k)} - Y^\ast )\right]\\
                           &= ||Y^{(k)} - Y^\ast ||_F^2 - ||A^{\dag}_{U_{i_k},:}A_{U_{i_k},:} (Y^{(k)} - Y^\ast )||_F^2.
\end{align*}
By Definition \ref{ARBK:definition_paving} and Lemma \ref{AY-norma1}, we have
\begin{align*}
||A^{\dag}_{U_{i_k},:}A_{U_{i_k},:} (Y^{(k)} - Y^\ast )||_F^2
\geq \sigma_{\min}^2(A^{\dag}_{U_{i_k},:})||A_{U_{i_k},:} (Y^{(k)} - Y^\ast )||_F^2
\geq \frac{1}{\beta_A} ||A_{U_{i_k},:} (Y^{(k)} - Y^\ast )||_F^2,
\end{align*}
where the second inequality follows from the fact that
\begin{align*}
  \sigma_{\min}^2(A^{\dag}_{U_{i_k},:}) = \sigma_{\max}^{-2}(A_{U_{i_k},:}) \geq 1/\beta_A.
\end{align*}
It implies that
\begin{align*}
E_k\left[||A^{\dag}_{U_{i_k},:}A_{U_{i_k},:} (Y^{(k)} - Y^\ast )||_F^2\right]
&\geq E\left[\frac{1}{\beta_A}||A_{U_{i_k},:} (Y^{(k)} - Y^\ast )||_F^2\right]\\
& = \frac{1}{s \beta_A} \sum_{U_{i_k} \in S} ||A_{U_{i_k},:} (Y^{(k)} - Y^\ast )||_F^2 \\
& = \frac{1}{s \beta_A} ||A(Y^{(k)} - Y^\ast )||_F^2 \\
&\geq \frac{\sigma_{\min}^2(A)}{s \beta_A}||Y^{(k)} - Y^\ast ||_F^2,
\end{align*}
where $E_k \left[ \cdot \right]$ denotes the expected value conditional on the first $k+1$ iterations for computing $Y^{(k+1)}$.  Therefore, we have
\begin{align*}
E_k\left[||Y^{(k+1)} - Y^\ast||_F^2\right]
& = ||Y^{(k)} - Y^\ast ||_F^2 - E_k\left[||A^{\dag}_{U_{i_k},:}A_{U_{i_k},:} (Y^{(k)} - Y^\ast )||_F^2\right]  \nonumber \\
&\leq ||Y^{(k)} - Y^\ast ||_F^2 - \frac{\sigma_{\min}^2(A)}{s \beta_A}||Y^{(k)} - Y^\ast ||_F^2 \nonumber \\
& = \widehat{\gamma} ||Y^{(k)} - Y^\ast ||_F^2.
\end{align*}
Since $Y^{(k+1)}$ is not only related to the previous approximation $Y^{(k)}$, but also to all the predecessors $Y^{(0)}, Y^{(1)},\cdots, Y^{(k)}$, taking full expectations instead of conditional expectations on both sides of the above formula yields
\begin{align*}
  E \left[||Y^{(k+1)} - Y^\ast||_F^2\right]  \leq  \widehat{\gamma}  E \left[ ||Y^{(k)} - Y^\ast ||_F^2 \right].
\end{align*}
By induction on the iteration index $k$, we straightforwardly obtain the estimate (\ref{E_Yk}). This completes the proof.
\hfill
\end{proof}

Now we turn to analyze the convergence property of the ARBK method.

\vskip 0.5ex
\begin{theorem}\label{Thm:ARBK}
Let the matrix equations in \eqref{MatrixAXB=F} be consistent,  and $S$ and $T$ be the row and column random partitions of $[m]$ and $[q]$ defined by Definitions \ref{ARBK:definition-RP} and \ref{ARBK:definition-RP2}, respectively. Starting from the initial matrices
 $X^{(0)}\in \mathcal{R}(I_{m\times n}, B)$ and $Y^{(0)}\in \mathcal{R}(A, I_{p\times q})$, the sequence $\left\{X^{(k)}\right\}_{k=0}^{\infty}$ generated by Algorithm \ref{alg:1} converges linearly to the solution $X^{\ast}$ in mean square form. Moreover, the solution error in expectation for the iteration sequence obeys
\begin{equation}\label{E_Xk}
E\left[||X^{(k+1)} - X^\ast||_F^2\right] \leq \overline{\gamma}_k \widetilde{\gamma}^{k+1} ||X^{(0)} - X^\ast ||_F^2,
\end{equation}
where the convergence factor $\widetilde{\gamma} = 1 -  \sigma_{\min}^2(B) / (t \beta_B)$ and the parameter
\begin{align*}
\overline{\gamma}_k = 1 + \frac{\sigma_{\max}^2(B)}{t \alpha_B} \sum\limits_{\ell=0}^{k}\widehat{\gamma}^{\ell+1}/\widetilde{\gamma}^{\ell+1}
\end{align*}
with $\widehat{\gamma}$ being defined by Lemma \ref{AY-norma2}.
\end{theorem}

\begin{proof}
According to iterative scheme on $X^{(k)}$ for $k=0,1,2,\cdots$, we have
\begin{align*}
X^{(k+1)} - X^{\ast} &= X^{(k)} - X^{\ast} + (Y^{(k+1)}_{:,V_{j_k}} - X^{(k)} B_{:,V_{j_k}}) B^{\dag}_{:,V_{j_k}}\\
                     &= X^{(k)} - X^{\ast} + (Y^{(k+1)}_{:,V_{j_k}} - Y^{\ast}_{:,V_{j_k}} + Y^{\ast}_{:,V_{j_k}}- X^{(k)} B_{:,V_{j_k}}) B^{\dag}_{:,V_{j_k}}\\
                     &= X^{(k)} - X^{\ast} - (X^{(k)} B_{:,V_{j_k}} - Y^{\ast}_{:,V_{j_k}}) B^{\dag}_{:,V_{j_k}} + (Y^{(k+1)}_{:,V_{j_k}} - Y^{\ast}_{:,V_{j_k}}) B^{\dag}_{:,V_{j_k}}.
\end{align*}
Since the matrix equation \eqref{AXB=F}  is consistent, it follows that $Y^{\ast}_{:,V_{j_k}} = X^{\ast}B_{:,V_{j_k}}$ by using $Y^{\ast} = X^{\ast}B$ and the first two terms of the previous expression can be further rewritten by
\begin{align*}
 X^{(k)} - X^{\ast} - (X^{(k)} B_{:,V_{j_k}} - Y^{\ast}_{:,V_{j_k}}) B^{\dag}_{:,V_{j_k}} = (X^{(k)} - X^{\ast})(I_{n} - B_{:,V_{j_k}}B^{\dag}_{:,V_{j_k}}).
\end{align*}
After an elementary calculation, it is easy to obtain that
\begin{align*}
||X^{(k+1)} - X^{\ast}||_F^2 &= ||(X^{(k)} - X^{\ast})(I_{n} - B_{:,V_{j_k}}B^{\dag}_{:,V_{j_k}})||_F^2 + ||(Y^{(k+1)}_{:,V_{j_k}} - Y^{\ast}_{:,V_{j_k}}) B^{\dag}_{:,V_{j_k}}||_F^2
\end{align*}
which comes from the fact that
\begin{align*}
& \quad \left\langle(X^{(k)} - X^{\ast})(I_{n} - B_{:,V_{j_k}}B^{\dag}_{:,V_{j_k}}),(Y^{(k+1)}_{:,V_{j_k}} - Y^{\ast}_{:,V_{j_k}}) B^{\dag}_{:,V_{j_k}}\right\rangle_F \\
&= {\rm Tr}\left\{(I_{n} - B_{:,V_{j_k}}B^{\dag}_{:,V_{j_k}})^T(X^{(k)} - X^{\ast})^T(Y^{(k+1)}_{:,V_{j_k}} - Y^{\ast}_{:,V_{j_k}}) B^{\dag}_{:,V_{j_k}}\right\} \\
&= {\rm Tr}\left\{B^{\dag}_{:,V_{j_k}}(I_{n} - B_{:,V_{j_k}}B^{\dag}_{:,V_{j_k}})(X^{(k)} - X^{\ast})^T(Y^{(k+1)}_{:,V_{j_k}} - Y^{\ast}_{:,V_{j_k}})\right\} \\
&= {\rm Tr}\left\{B^{\dag}_{:,V_{j_k}} - B^{\dag}_{:,V_{j_k}}B_{:,V_{j_k}}B^{\dag}_{:,V_{j_k}})(X^{(k)} - X^{\ast})^T(Y^{(k+1)}_{:,V_{j_k}} - Y^{\ast}_{:,V_{j_k}})\right\} \\
&= 0
\end{align*}
by using the properties of pseudoinverse. Then, we have
\begin{align}\label{eq:ARBK_error1}
||X^{(k+1)} - X^{\ast}||_F^2
&= ||(X^{(k)} - X^{\ast})(I_{n} - B_{:,V_{j_k}}B^{\dag}_{:,V_{j_k}})||_F^2 + ||(Y^{(k+1)}_{:,V_{j_k}} - Y^{\ast}_{:,V_{j_k}}) B^{\dag}_{:,V_{j_k}}||_F^2 \notag \\
&= ||X^{(k)} - X^{\ast}||_F^2 - ||(X^{(k)} - X^{\ast})B_{:,V_{j_k}}B^{\dag}_{:,V_{j_k}}||_F^2 + ||(Y^{(k+1)}_{:,V_{j_k}} - Y^{\ast}_{:,V_{j_k}}) B^{\dag}_{:,V_{j_k}}||_F^2.
\end{align}

Let $E_k$ denote the expected value conditional on the first $k+1$ iterations in Algorithm \ref{alg:1}, that is
\begin{align*}
  E_k\left[\cdot\right] = E\left[ \cdot~ |~ U_{i_0}, V_{j_0},U_{i_1}, V_{j_1},\cdots,U_{i_k}, V_{j_k}\right],
\end{align*}
where $i_\ell$ and $j_\ell$ are independent of each other and selected uniformly at the $\ell$th iteration for $\ell=0,1,\cdots,k$. Let
the conditional expectations with respect to $U_{i_\ell}$ be
\begin{align*}
  E_{k,1}\left[\cdot\right] = E\left[ \cdot~ |~ U_{i_0}, V_{j_0},U_{i_1}, V_{j_1},\cdots,U_{i_{k-1}}, V_{j_{k-1}},V_{j_k}\right]
\end{align*}
and with respect to $V_{j_\ell}$ be
\begin{align*}
  E_{k,2}\left[\cdot\right] = E\left[ \cdot~ |~ U_{i_0}, V_{j_0},U_{i_1}, V_{j_1},\cdots,U_{i_{k-1}}, V_{j_{k-1}},U_{i_k}\right].
\end{align*}
According to the law of total expectation, we have
\begin{align*}
  E_k\left[\cdot\right] = E_{k,1}\left[  E_{k,2}\left[\cdot\right]  \right] = E_{k,2}\left[  E_{k,1}\left[\cdot\right]  \right].
\end{align*}
Then, the expectation of the second term in formula \eqref{eq:ARBK_error1} can be computed by
\begin{align*}
E_{k}\left[||(X^{(k)} - X^{\ast})B_{:,V_{j_k}}B^{\dag}_{:,V_{j_k}}||_F^2\right]
& = E_{k,1}\left[ E_{k,2}\left[||(X^{(k)} - X^{\ast})B_{:,V_{j_k}}B^{\dag}_{:,V_{j_k}}||_F^2\right]\right]\\
& \geq E_{k,1}\left[\frac{1}{\beta_B} E_{k,2}\left[||(X^{(k)} - X^{\ast})B_{:,V_{j_k}}||_F^2\right]\right]\\
&= \frac{1}{t \beta_B}E_{k,1}\left[ \sum\limits_{ V_{j_k} \in T }||(X^{(k)} - X^{\ast})B_{:,V_{j_k}}||_F^2\right]\\
&= \frac{1}{t \beta_B} E_{k,1}\left[ ||(X^{(k)} - X^{\ast})B||_F^2\right]\\
&\geq \frac{\sigma_{\min}^2(B)}{t \beta_B} E_{k,1}\left[ ||X^{(k)} - X^\ast ||_F^2\right],
\end{align*}
where the first inequality follows from the fact that
\begin{align*}
||(X^{(k)} - X^{\ast})B_{:,V_{j_k}}B^{\dag}_{:,V_{j_k}}||_F^2 &\geq \sigma_{\min}^2(B^{\dag}_{:,V_{j_k}})||(X^{(k)} - X^{\ast})B_{:,V_{j_k}}||_F^2\\
&= \frac{1}{\sigma_{\max}^2(B_{:,V_{j_k}})}||(X^{(k)} - X^{\ast})B_{:,V_{j_k}}||_F^2\\
&\geq \frac{1}{\beta_B}||(X^{(k)} - X^{\ast})B_{:,V_{j_k}}||_F^2
\end{align*}
and the last inequality is from Definition \ref{ARBK:definition_pavingB}. By taking the expectation of both sides, we have
\begin{align}\label{ARBK:ThmEq1}
  E\left[||(X^{(k)} - X^{\ast})B_{:,V_{j_k}}B^{\dag}_{:,V_{j_k}}||_F^2\right]
  \geq \frac{\sigma_{\min}^2(B)}{t \beta_B} E \left[ ||X^{(k)} - X^\ast ||_F^2\right].
\end{align}

Similarly, the expectation of the third term in formula \eqref{eq:ARBK_error1} satisfies that
\begin{align*}
E_{k}\left[||(Y^{(k+1)}_{:,V_{j_k}} - Y^{\ast}_{:,V_{j_k}}) B^{\dag}_{:,V_{j_k}}||_F^2\right]
& = E_{k,1}\left[ E_{k,2}\left[||(Y^{(k+1)}_{:,V_{j_k}} - Y^{\ast}_{:,V_{j_k}}) B^{\dag}_{:,V_{j_k}}||_F^2\right] \right]\\
&\leq E_{k,1}\left[ \frac{1}{\alpha_B }E_{k,2}\left[||Y^{(k+1)}_{:,V_{j_k}} - Y^{\ast}_{:,V_{j_k}}||_F^2\right]\right]\\
&= \frac{1}{t \alpha_B} E_{k,1}\left[ ||Y^{(k+1)} - Y^{\ast}||_F^2\right].
\end{align*}
where the first inequality comes from the fact that
\begin{align*}
||(Y^{(k+1)}_{:,V_{j_k}} - Y^{\ast}_{:,V_{j_k}}) B^{\dag}_{:,V_{j_k}}||_F^2 &\leq \sigma_{\max}^2(B^{\dag}_{:,V_{j_k}})||Y^{(k+1)}_{:,V_{j_k}} - Y^{\ast}_{:,V_{j_k}}||_F^2\\
&= \frac{1}{\sigma_{\min}^2(B_{:,V_{j_k}})}||Y^{(k+1)}_{:,V_{j_k}} - Y^{\ast}_{:,V_{j_k}}||_F^2\\
&\leq \frac{1}{\alpha_B} ||Y^{(k+1)}_{:,V_{j_k}} - Y^{\ast}_{:,V_{j_k}}||_F^2.
\end{align*}
By taking the expectation of both sides again, we have
\begin{align}\label{ARBK:ThmEq2}
E\left[||(Y^{(k+1)}_{:,V_{j_k}} - Y^{\ast}_{:,V_{j_k}}) B^{\dag}_{:,V_{j_k}}||_F^2\right]
\leq \frac{1}{t \alpha_B}  E \left[ ||Y^{(k+1)} - Y^{\ast}||_F^2\right]
\leq \frac{1}{t \alpha_B} \widehat{\gamma}^{k+1} ||Y^{(0)} - Y^\ast||_F^2,
\end{align}
where the last inequality is from Lemma  \ref{AY-norma2}.

Based on formulas \eqref{ARBK:ThmEq1} and \eqref{ARBK:ThmEq2}, the expectation of $||X^{(k+1)} - X^{\ast}||_F^2$ is bounded by
 \begin{align}\label{EXk}
E\left[||X^{(k+1)} - X^{\ast}||_F^2\right]
&\leq \left( 1 - \frac{\sigma_{\min}^2(B)}{t \beta_B} \right) E\left[||X^{(k)} - X^\ast ||_F^2\right]
+ \frac{1}{t \alpha_B} \widehat{\gamma}^{k+1} ||Y^{(0)} - Y^\ast||_F^2 \nonumber  \\
& = \widetilde{\gamma} E\left[||X^{(k)} - X^\ast ||_F^2\right]
+ \frac{\widehat{\gamma}^{k+1}}{t \alpha_B}  ||Y^{(0)} - Y^\ast||_F^2 \nonumber  \\
&\leq \widetilde{\gamma}
\left(\widetilde{\gamma} E\left[||X^{(k-1)} - X^\ast ||_F^2\right] + \frac{ \widehat{\gamma}^{k}}{t \alpha_B}||Y^{(0)} - Y^\ast ||_F^2\right)
+ \frac{ \widehat{\gamma}^{k+1}}{t \alpha_B}||Y^{(0)} - Y^\ast ||_F^2 \nonumber \\
&\leq \cdots \nonumber \\
&\leq \widetilde{\gamma}^{k+1} ||X^{(0)} - X^\ast ||_F^2 + \frac{1}{t \alpha_B }\sum\limits_{\ell=0}^{k}\widehat{\gamma}^{\ell+1}\widetilde{\gamma}^{k-\ell}||Y^{(0)} - Y^\ast ||_F^2.
\end{align}
By using $Y^{\ast} = X^{\ast}B$, we have
\begin{align}\label{X-Y}
||Y^{(0)} - Y^{\ast}||_F^2 = ||(X^{(0)} - X^{\ast})B||_F^2 \leq \sigma_{\max}^2(B)||X^{(0)} - X^{\ast}||_F^2.
\end{align}
Consequently, it follows that
\begin{align*}
E\left[||X^{(k+1)} - X^{\ast}||_F^2\right] &\leq \widetilde{\gamma}^{k+1} ||X^{(0)} - X^\ast ||_F^2 + \frac{\sigma_{\max}^2(B)}{t \alpha_B}\sum\limits_{\ell=0}^{k}\widehat{\gamma}^{\ell+1}\widetilde{\gamma}^{k-\ell}||X^{(0)} - X^{\ast}||_F^2 \nonumber\\
&= \left(1 + \frac{\sigma_{\max}^2(B)}{t \alpha_B} \sum\limits_{\ell=0}^{k}\widehat{\gamma}^{\ell+1}/\widetilde{\gamma}^{\ell+1}\right)\widetilde{\gamma}^{k+1}||X^{(0)} - X^\ast ||_F^2
\end{align*}
by substituting \eqref{X-Y} into \eqref{EXk}. This completes the proof.
\hfill
\end{proof}

\vskip 0.5ex
\begin{remark}
We notice that the criterion in CME-RK \cite{23XBL} is equivalent to a uniform sampling if the matrices $A$ and $B$ are proactively scaled with two diagonal matrices that normalize the Euclidean norms of all of its rows and columns to be a same constant, respectively.
The convergence result in Remark \ref{ARBK:remark_CMERKcon} for the CME-RK method reduces to
\begin{align*}
E\left[||X^{(k+1)} - X^\ast||_F^2\right]
& \leq
\left( 1 + \frac{\sigma_{\max}^2(B)}{m}  \sum\limits_{\ell=0}^{k} \widehat{\rho}_1^{\ell+1}/\widehat{\rho}_2^{\ell+1} \right) \widehat{\rho}_2^{k+1}
||X^{(0)} - X^\ast ||_F^2 \\
& = \widehat{\rho}_2^{k+1} ||X^{(0)} - X^\ast ||_F^2 + \frac{\sigma_{\max}^2(B)}{m}\sum\limits_{\ell=0}^{k}\widehat{\rho}_1^{\ell+1}\widehat{\rho}_2^{k-\ell}||X^{(0)} - X^{\ast}||_F^2,
\end{align*}
where the convergence factors $\widehat{\rho}_1 = 1 -  \sigma_{\min}^2(A)/m$ and $\widehat{\rho}_2 = 1 -  \sigma_{\min}^2(B)/q$. It implies that the block technique will further accelerate the convergence of the CME-RK method if the ARBK method has the well-defined paving. For example, $\beta_A \leq m/s$ and $\beta_B \leq q/t$. These two hypotheses are easily obtained by setting $\beta_A = \max\limits_{U\in S} \left\{ \| A_{U,:} \|_F^2 \right\} =|U|$ and $\beta_B = \max\limits_{V\in T} \left\{ \| B_{:,V} \|_F^2 \right\}=|V|$.
\end{remark}

\begin{remark}
For $k=0,1,2,\cdots$, the solution error sequence in expectation for the GRBK iteration admits
\begin{align*}
E\left[||X^{(k+1)} - X^\ast||_F^2\right] \leq
\left(
1 - \frac{\sigma_{\min}^2(A)}{\| A\|^2_F \beta_{\max}^2(A)}
    \frac{\sigma_{\min}^2(B)}{\| B\|^2_F \beta_{\max}^2(B)}
\right)^{k+1}
 ||X^{(0)} - X^\ast ||_F^2,
\end{align*}
where
\begin{align*}
  \beta_{\max}(A) = \max\limits_{U\in S} \left\{ \frac{\sigma_{\max}(A_{U,:})}{\| A_{U,:}\|_F} \right\}
  ~~ and ~~
  \beta_{\max}(B) = \max\limits_{V\in T} \left\{ \frac{\sigma_{\max}(B_{:,V})}{\| B_{:,V}\|_F} \right\};
\end{align*}
see also Equations (2.9) and (2.10) in \cite{22Niu}. Let every block in $S$ and $T$ have equal size denoted by $|U|$ and $|V|$, respectively. By normalizing the Euclidean norms of $\left\{A_{1,:}, A_{2,:},\cdots,A_{m,:}\right\}$ and $\left\{B_{:,1}, B_{:,2},\cdots,B_{:,q}\right\}$, it follows that
\begin{align*}
E\left[||X^{(k+1)} - X^\ast||_F^2\right] \leq
\left(
1 - \frac{\sigma_{\min}^2(A)}{ s\beta_A }
    \frac{\sigma_{\min}^2(B)}{ t\beta_B }
\right)^{k+1}
 ||X^{(0)} - X^\ast ||_F^2
\end{align*}
according to $s=m/|U|$, $t=q/|V|$, $\| A_{U,:}\|^2_F = |U|$, and $\| B_{:,V}\|^2_F = |V|$ for $U\in S$ and $V\in T$, respectively. Since the convergence rate of the ARBK method is mainly determined by $\widetilde{\gamma}$, we omit the term $\overline{\gamma}_k$ in \eqref{E_Xk} to simplify the analysis process. It follows that
\begin{align*}
  \widetilde{\gamma} < 1 - \frac{\sigma_{\min}^2(A)}{ s\beta_A } \frac{\sigma_{\min}^2(B)}{ t\beta_B }
\end{align*}
from the fact that
\begin{align*}
  0<
  \frac{\sigma_{\min}^2(A)}{ s\beta_A }
   = \frac{\sigma_{\min}^2(A)}{m/|U|}\frac{1}{\max\limits_{U\in S}\left\{ \sigma_{\max}^2(A_{U,:}) \right\}}
   =\frac{||A_{U,:} ||_F^2}{||A ||_F^2} \frac{\sigma_{\min}^2(A)}{\sigma_{\max}^2(A_{U,:})} <1,
\end{align*}
which indicates that the ARBK method will have a faster convergence factor than the GRBK method \cite{22Niu}. Later, the acceleration will become apparent in Section \ref{AXB=F:Numer}.
\end{remark}

\section{Numerical Experiments}\label{AXB=F:Numer}
In this section, we present some numerical experiments to verify the effectiveness of our algorithm performed by using MATLAB (version R2021a) on a notebook PC with Intel(R) Core(TM) i5-10210U CPU 1.60GHz-2.10GHz.

In all examples, the coefficient matrices $A$ and $B$ are given, the exact solution is generated by using the MATLAB built-in function, e.g., $X^{\ast} = {\sf ones}(n,p)$, and the right-hand side matrix is given by $F=AX^{\ast}B$. We compare the performance of the tested methods in terms of iteration number (denoted as IT), computing time in seconds (denoted as CPU), and relative solution error (denoted by RSE), where the last term is defined by ${\rm RSE}= ||X^{(k)} - X^{\ast}||_F/||X^{\ast}||_F$ with the iteration starting from the zero matrix. The CPU and IT are the arithmetical averages of the elapsed CPU times and the required iteration steps concerning $20$ times repeated runs of the corresponding method. The Moore-Penrose pseudoinverse in the block iterative methods is generated by the MATLAB built-in function {\sf pinv}. Every row and column index sets are defined by Definitions \eqref{ARBK:definition-RP} and \eqref{ARBK:definition-RP2}, respectively.

For the first example, we demonstrate the computational behavior of the ARBK method with respect to different block-size. The coefficient matrix is the synthetic data generated as follows.

\vskip 1.25ex
\begin{example}\label{ARBK:example1}
For given $m$, $n$, $p$ and $q$, we consider two dense coefficient matrices, which are randomly generated by the MALTAB built-in function, e.g., $A = {\sf randn}(m,n)$ and $B = {\sf randn}(p,q)$.
\end{example}
\vskip 1.25ex

This example gives us many flexibilities to adjust the input parameters. We consider four cases by setting $m=q=350,450,550,650$ and $n=p=210,270,330,390$ and try to find the appropriate block sizes for the ARBK method. Seven pairs of block-sizes $(\tau_A, \tau_B)=(m/s, q/t)$ are considered, i.e., $(\tau_A, \tau_B)=(10,10)$, $(20,20)$, $(30,30)$, $(40, 40)$, $(50,50)$, $(60,60)$, and $(80,80)$. Correspondingly, the ARBK method is denoted by ARBK($\tau_A$, $\tau_B$).

The average convergence behaviors of relative solution error versus the number of iteration step and computing time given by ARBK($\tau_A$, $\tau_B$) are shown in Figure \ref{fig:VariousBlockSizeEx1}. In this figure, the heavy line represents the median performance and the shaded region spans the minimum to the maximum value across all trials. It is observed that the convergence of the ARBK method becomes faster with an increase in block-size and then slows down after reaching the fastest rate in terms of computing time. The results show that the block-size pairs $(\tau_A, \tau_B)=(30,30)$ and $(50,50)$ are the favorable choices for the ARBK method leading to a satisfactory convergence. Not especially specified, we will adopt these two parameter  selection approaches for the ARBK method in the following numerical test.

\begin{figure}[!htb]
\centering
    \subfigure[$m=q=350$, $n=p=210$]{
		\includegraphics[width=0.5\textwidth]{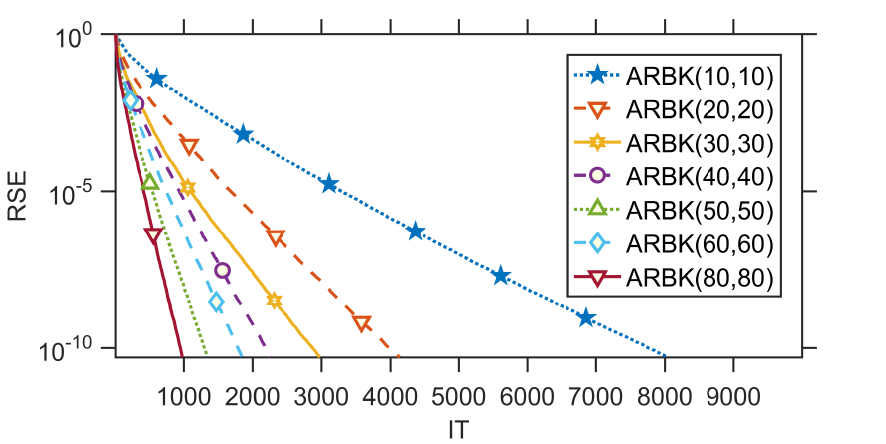}}
    \hspace{-0.7cm}
    \subfigure[$m=q=350$, $n=p=210$]{
		\includegraphics[width=0.5\textwidth]{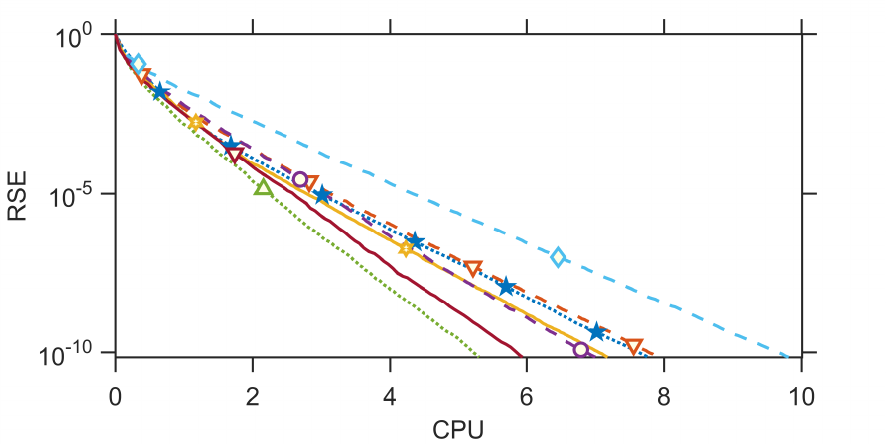}}
    \hspace{2cm}
    \subfigure[$m=q=450$, $n=p=270$]{
		\includegraphics[width=0.5\textwidth]{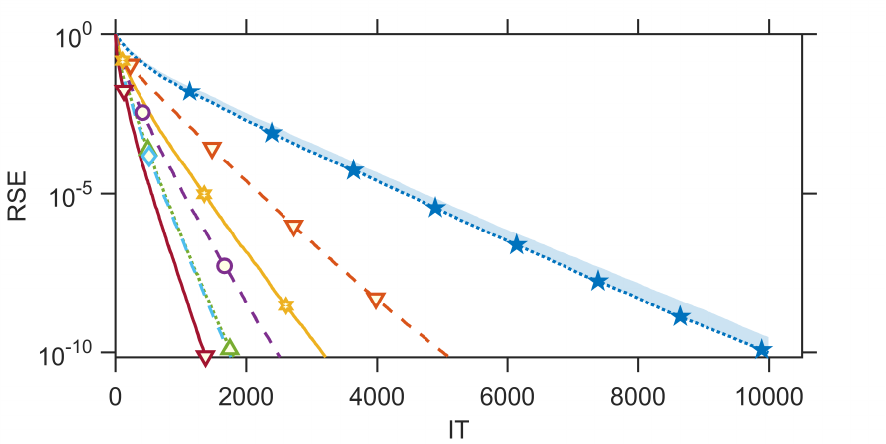}}
    \hspace{-0.7cm}
    \subfigure[$m=q=450$, $n=p=270$]{
		\includegraphics[width=0.5\textwidth]{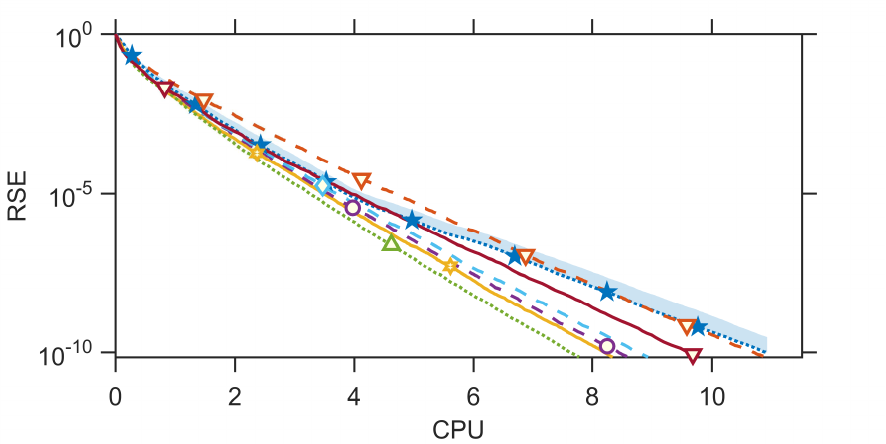}}
    \hspace{2cm}
    \subfigure[$m=q=550$, $n=p=330$]{
		\includegraphics[width=0.5\textwidth]{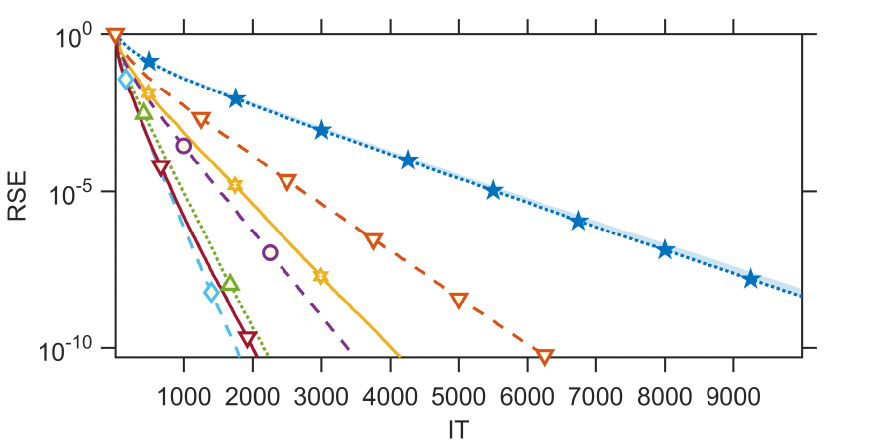}}
    \hspace{-0.7cm}
    \subfigure[$m=q=550$, $n=p=330$]{
		\includegraphics[width=0.5\textwidth]{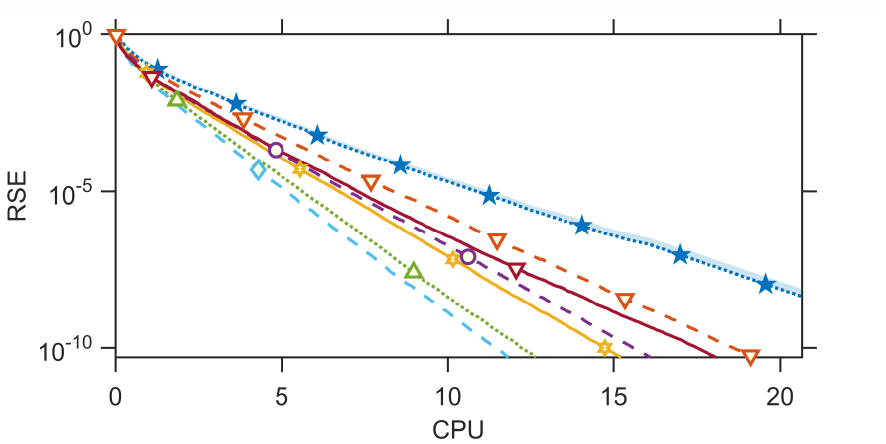}}
    \hspace{2cm}
    \subfigure[$m=q=650$, $n=p=390$]{
		\includegraphics[width=0.5\textwidth]{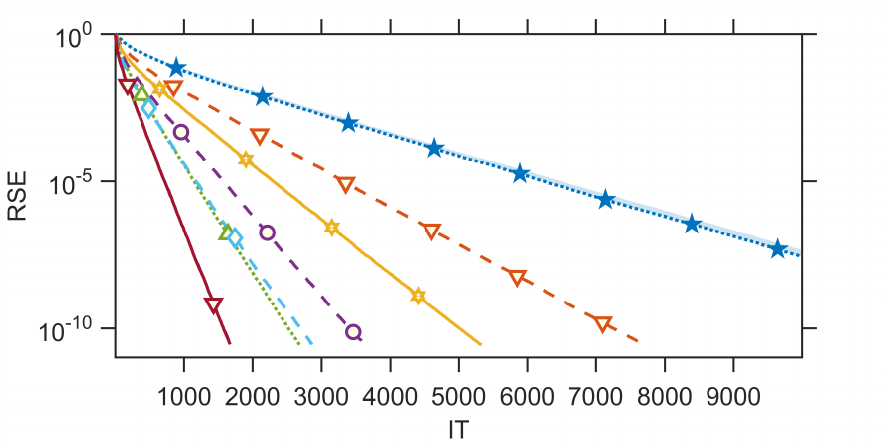}}
    \hspace{-0.7cm}
    \subfigure[$m=q=650$, $n=p=390$]{
		\includegraphics[width=0.5\textwidth]{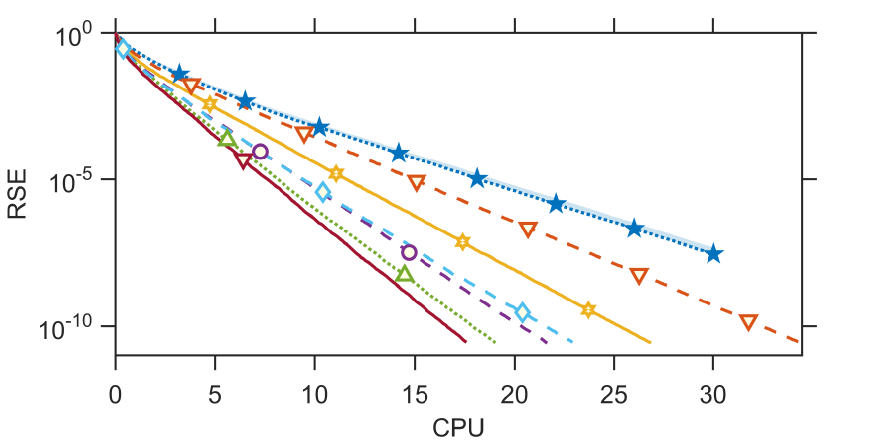}}
\caption{The convergence behaviors of RSE versus IT and CPU given by the ARBK method for Example \ref{ARBK:example1} with various sizes of the coefficient matrices.}
\label{fig:VariousBlockSizeEx1}
\end{figure}

\vskip 1.25ex
\begin{example}\label{ARBK:example2}
Like  \cite{Du19}, Du constructed a dense  matrix $R$  by $R = U\Sigma V^T\in \Rc^{x_1\times x_2}$,
where the entries of $U\in \Rc^{x_1\times r}$ and $V\in \Rc^{x_2\times r}$ are generated from a standard normal distribution,
and then their columns are orthonormalized. The matrix $\Sigma$ is an $r$-by-$r$ diagonal matrix whose first $r$-2 diagonal entries are uniformly distributed numbers in $[\sigma_2, \sigma_1]$, and the last two diagonal entries are $\sigma_2$ and $\sigma_1$. The output matrix is abbreviated as $ R =   {\sf{Smatrix}} (x_1, x_2,r,\sigma_1,\sigma_2)$.
\end{example}
\vskip 1.25ex

For the second example, we compare the performances of the ARBK, CME-RK \cite{23XBL}, and GRBK \cite{22Niu} methods, which are all row-action methods. We conduct our evaluation on a synthetic data set customized to yield different specific instances by manipulating the input parameters. Four kinds of matrix equations are considered. The size (i.e., $m$, $n$, $p$, and $q$), rank (i.e., $r_{A}$ and $r_{B}$), and Euclidean condition number (i.e., $\sigma_{A_1}/\sigma_{A_2}$ and  $\sigma_{B_1}/\sigma_{B_2}$) of the test matrices are given in Table \ref{ARBK:tab:ex2-Parameters}. The experiments are terminated once the relative solution error is less than $5\times 10^{-2}$ or the number of iteration step exceeds $10^5$.

We show the numerical results of relative solution error, iteration number, and computing time in Table \ref{ARBK:tab:ex2-NumericalResults}. For the block iterative methods, i.e., ARBK and GRBK \cite{22Niu}, the block sizes are set to $(30,30)$ and $(50,50)$.  It can be seen that both ARBK($30$, $30$) and ARBK($50$, $50$) successfully compute an approximate solution for all cases, but the CME-RK, GRBK($30$, $30$), and GRBK($50$, $50$) methods fail for the Cases (I, II, IV), (II, III), and III, respectively,  due to the number of iteration steps exceeds $10^{5}$. For all convergent cases, the iteration counts and computing times of ARBK($30$, $30$) and ARBK($50$, $50$) are appreciably smaller than those of CME-RK,  GRBK($30$, $30$), and GRBK($50$, $50$). Hence, the ARBK method outperforms the CME-RK and GRBK methods in terms of iteration counts as well as computing times.

\begin{table}[!ht]
 \normalsize
\caption{The parameters used in $A =  {\sf{Smatrix}} (m, n,r_{A},\sigma_{A_1},\sigma_{A_2})$ and $B =  {\sf{Smatrix}} (p, q,r_{B},\sigma_{B_1},\sigma_{B_2})$ for Example \ref{ARBK:example2}.}
\centering
\begin{tabular}{|c| ccccc| ccccc|}
\cline{1-11}
Cases&$m$&$n$&$r_{A}$&$\sigma_{A_1}$&$\sigma_{A_2}$&$p$&$q$&$r_{B}$&$\sigma_{B_1}$&$\sigma_{B_2}$\\
\cline{1-11}
Case I   & $1000$ &$100$&$100$&$10$&$0.1$ &$1000$ &$100$&$100$&$10$&$0.1$\\
Case II  & $2000$ &$100$&$100$&$10$&$0.1$ &$2000$ &$100$&$100$&$10$&$0.1$\\
Case III & $4000$&$200$&$200$&$10$&$0.1$ &$4000$&$200$&$200$&$10$&$0.1$\\
Case IV & $10000$&$200$&$200$&$10$&$0.1$ &$10000$&$200$&$200$&$10$&$0.1$\\
\cline{1-11}
\end{tabular}
\label{ARBK:tab:ex2-Parameters}
\end{table}

\begin{table}[!ht]
 \normalsize
\caption{The numerical results of RSE, IT, and CPU for the CME-RK,  GRBK($30$, $30$),  GRBK($50$, $50$), ARBK($30$, $30$), and ARBK($50$, $50$)  methods, where the parameters used in $A$ and $B$ for Example \ref{ARBK:example2} are from Table \ref{ARBK:tab:ex2-Parameters}.}
\centering
\begin{tabular}{|c|ccc| ccc|}
\cline{1-7}
&\multicolumn{3}{|l|}{Case I} &\multicolumn{3}{|l|}{Case II}  \\
\cline{1-7}
Methods&RSE&IT&CPU&RSE&IT&CPU\\
\cline{1-7}
CME-RK&$>8.23\times 10^{-2}$&$>10^5$&$>23.34$&$>1.21\times 10^{-1}$&$>10^5$&$>78.88$\\
GRBK($30$, $30$)&$5.00\times 10^{-2}$&$94049.6$&$103.61$&$>5.26\times 10^{-2}$&$>10^5$&$>104.11$\\
GRBK($50$, $50$)&$5.00\times 10^{-2}$&$8293.2$&$17.98$&$5.00\times 10^{-2}$&$14311.4$&$29.62$\\
ARBK($30$, $30$)&$5.00\times 10^{-2}$&$3279.4$&$7.65$&$5.00\times 10^{-2}$&$7251.0$&$20.33$\\
ARBK($50$, $50$)&$5.00\times 10^{-2}$&$1135.2$&$3.79$&$5.00\times 10^{-2}$&$2567.2$&$11.38$\\
\cline{1-7}
&\multicolumn{3}{|l|}{Case III} &\multicolumn{3}{|l|}{Case IV}  \\
\cline{1-7}
Methods&RSE&IT&CPU&RSE&IT&CPU\\
\cline{1-7}
CME-RK           &$>1.03\times 10^{-1}$&$>10^5$&$>386.53$&$5.05\times 10^{-2}$&$98963.2$&$942.54$\\
GRBK($30$, $30$) &$>9.60\times 10^{-2}$&$>10^5$&$>158.49$&$>6.09\times 10^{-2}$&$>10^5$&$>153.94$\\
GRBK($50$, $50$) &$>5.16\times 10^{-2}$&$>10^5$&$>296.61$&$5.00\times 10^{-2}$&$33919.6$&$105.34$\\
ARBK($30$, $30$) &$5.00\times 10^{-2}$&$7458.5$&$57.61$&$5.00\times 10^{-2}$&$2484.6$&$42.26$\\
ARBK($50$, $50$) &$5.00\times 10^{-2}$&$3685.5$&$40.85$&$5.00\times 10^{-2}$&$1198.4$&$27.05$\\
\cline{1-7}
\end{tabular}
\label{ARBK:tab:ex2-NumericalResults}
\end{table}

%
%
%
%


\vskip 1.25ex

\begin{example}\label{ARBK:Ex:CAGD-AXB=C}
Assume that there are $m \times q$ data points to be fitted and sampled from the following surfaces whose coordinates are given by
\vskip 0.45ex
\begin{enumerate}[$\diamond$]
\setlength{\itemindent}{0.5cm}
\addtolength{\itemsep}{-0.1em} 
\item Data points 1: $x=s$, $y=t$, and $z=s^3 - 3st^2$ for $1/2 \leq t \leq 1$ and $0 \leq s \leq 2 \pi$; \vskip 0.45ex
\item Data points 2: $x=s$, $y=t$, and $z=(s-s^3-t^5)e^{-s^2-t^2}$ for $-1\leq s,t \leq 1$;
\end{enumerate}
\vskip 0.45ex
see also http://paulbourke.net/geometry/. As other researchers do, we first assign two parameter sequences $\nu_1$ and $\nu_2$, and two knot vectors $\mu_1$ and $\mu_2$ of cubic B-spline basis, whose formulations can be  respectively referred by Equations (9.5) and (9.69) in the book \cite{97PT}. Then, the collocation matrices are generated by using the MATLAB built-in function such as $A={\sf spcol}(\nu_1, 4, \mu_1)$ and $B^T={\sf spcol}(\nu_2, 4, \mu_2)$.
\end{example}

\vskip 1.25ex

The least squares progressive iterative approximation (LSPIA) \cite{14DL} is a famous data fitting method in computer-aided geometry design. In the following, we will compare the performance of our method with LSPIA for surface fitting. In this example, the collection matrices are sparse and of full-column rank.

We first list the numerical results, including the  relative solution error, the number of iteration steps, and the computing time, with different sizes of data points in Table \ref{ARBK:tab:Ex3-1}. It is observed that the relative solution errors of the tested methods are analogous while ARBK($50$, $50$) takes  much fewer iteration counts and computing times than LSPIA  for all cases. It indicates that ARBK($50$, $50$) is more efficient than LSPIA. Then, in Figure \ref{ARBK:fig:Ex3-1}, we draw the two initial data points with size $m=q=150$ in Example \ref{ARBK:Ex:CAGD-AXB=C}. Accordingly, we also respectively display the final cubic B-spline LSPIA and ARBK($50$, $50$) fitting surfaces in Figures  \ref{ARBK:fig:Ex3-21} and \ref{ARBK:fig:Ex3-22}. It shows that the proposed method approximates the given data points accurately.

\begin{table}[!htb]
 \normalsize
\caption{The numerical results of RSE, IT, and CPU for LSPIA and ARBK($50$, $50$) in Example  \ref{ARBK:Ex:CAGD-AXB=C} for various $m$ and $n$ with $q=m$ and $p=n$.}
\centering
\begin{tabular}{ |c|ccccc|}
\cline{1-6}
\textbf{Surface~1}& ($m$, $n$)        & (150, 50) & (350, 50)& (600, 200) & (1000, 200)\\
\cline{1-6}
LSPIA  &{RSE}&$4.99\times 10^{-2}$& $4.98\times 10^{-2}$ & $4.99\times 10^{-2}$ &$4.99\times 10^{-2}$\\
       &{IT}        &505  &505  & 582  & 556\\
       &{CPU}       &0.29 &1.01 & 16.35&44.17\\
\cline{1-6}
ARBK($50$, $50$) &{RSE}&$8.57\times 10^{-2}$& $2.83\times 10^{-2}$ & $1.98\times 10^{-2}$ &$1.99\times 10^{-2}$\\
       &{IT}        &12.8& 85.0  & 86.2 &465.8\\
       &{CPU}       &0.03& 0.34  & 1.02 &6.46\\
\cline{1-6}
\textbf{Surface~2}& ($m$, $n$)        & (150, 50) & (350, 50)& (600, 200) & (1000, 200)\\
\cline{1-6}
LSPIA  &{RSE}&$4.99\times 10^{-2}$& $5.00\times 10^{-2}$ & $4.89\times 10^{-2}$ &$4.99\times 10^{-2}$\\
       &{IT}        &633  &636  & 676  & 644\\
       &{CPU}       &0.31 &1.17 & 18.39&50.07\\
\cline{1-6}
ARBK($50$, $50$) &{RSE}&$1.11\times 10^{-2}$& $1.89\times 10^{-2}$ & $1.38\times 10^{-2}$ &$3.81\times 10^{-2}$\\
       &{IT}        &10.4& 123.8  & 75.4 &356.2\\
       &{CPU}       &0.03& 0.33  & 0.81 &4.71\\
\cline{1-6}
\end{tabular}
\label{ARBK:tab:Ex3-1}
\end{table}

\begin{figure}[!htb]
\centering
    \subfigure{
    \includegraphics[width=0.48\textwidth]{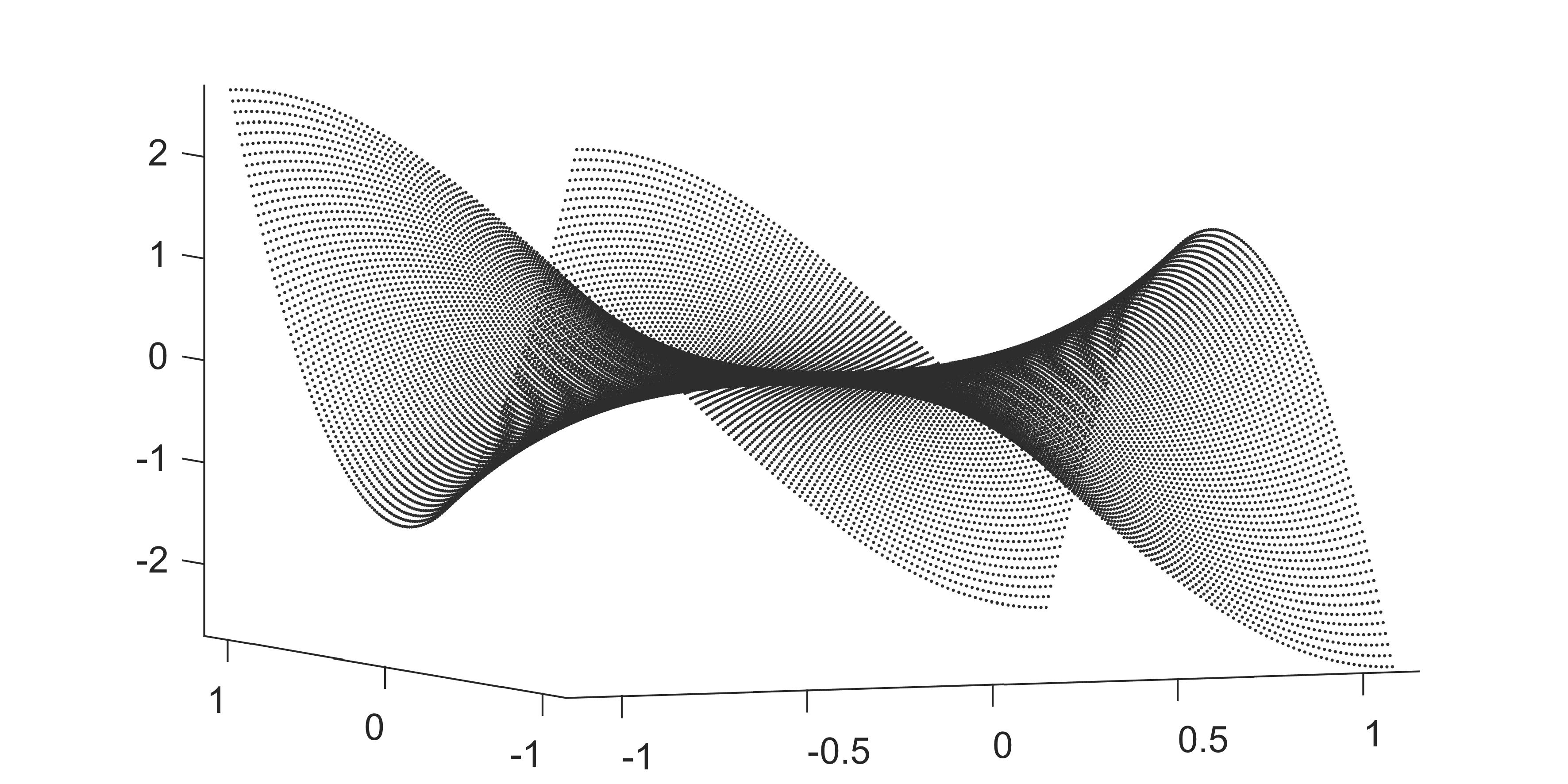}}
    \subfigure{
    \includegraphics[width=0.48\textwidth]{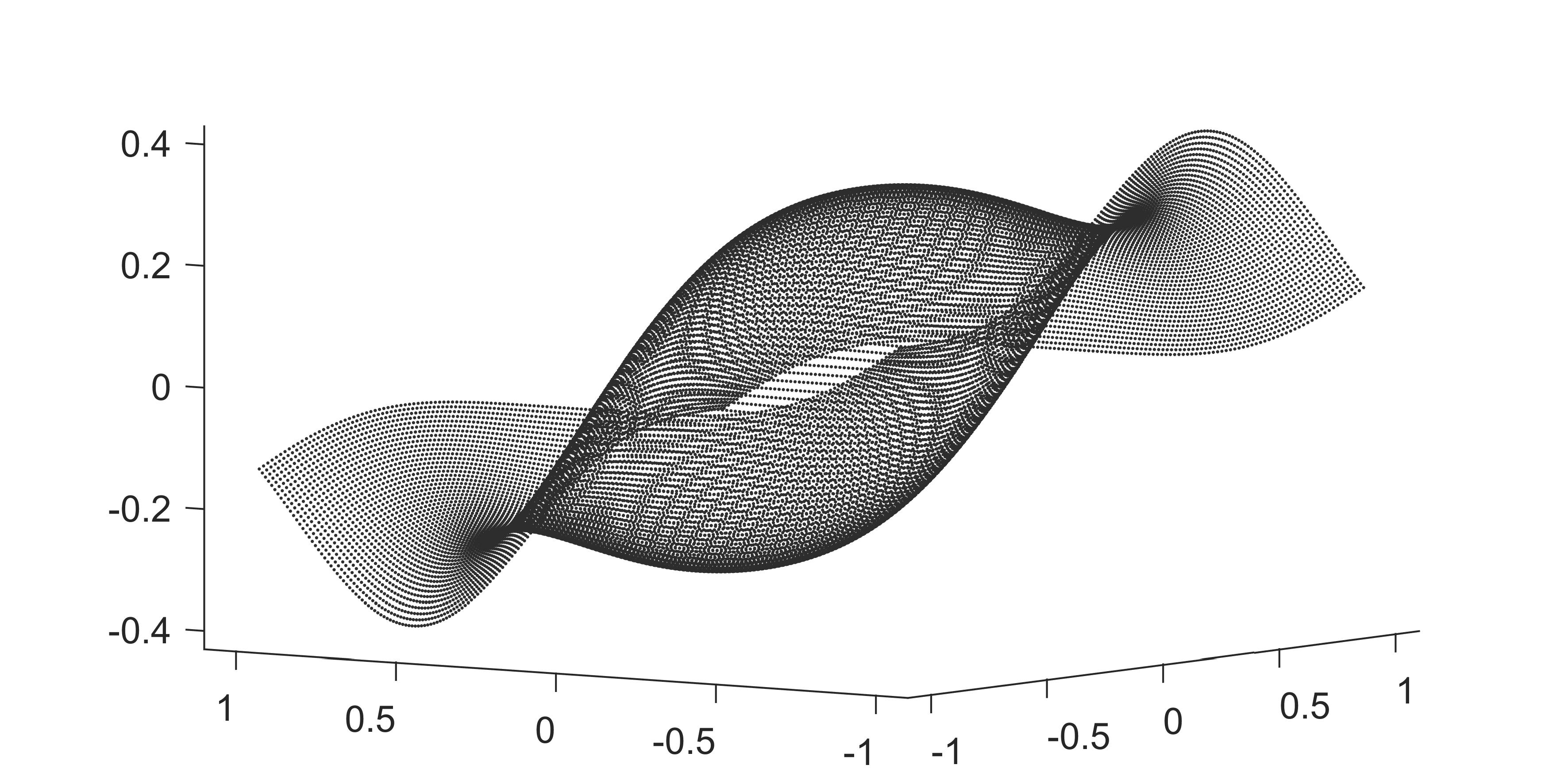}
    }
\caption{The $150 \times 150$ initial data points $1$ (left) and $2$ (right) in Example \ref{ARBK:Ex:CAGD-AXB=C}.}
\label{ARBK:fig:Ex3-1}
\end{figure}

\begin{figure}[!htb]
\centering
    \subfigure{
	     \includegraphics[width=0.48\textwidth]{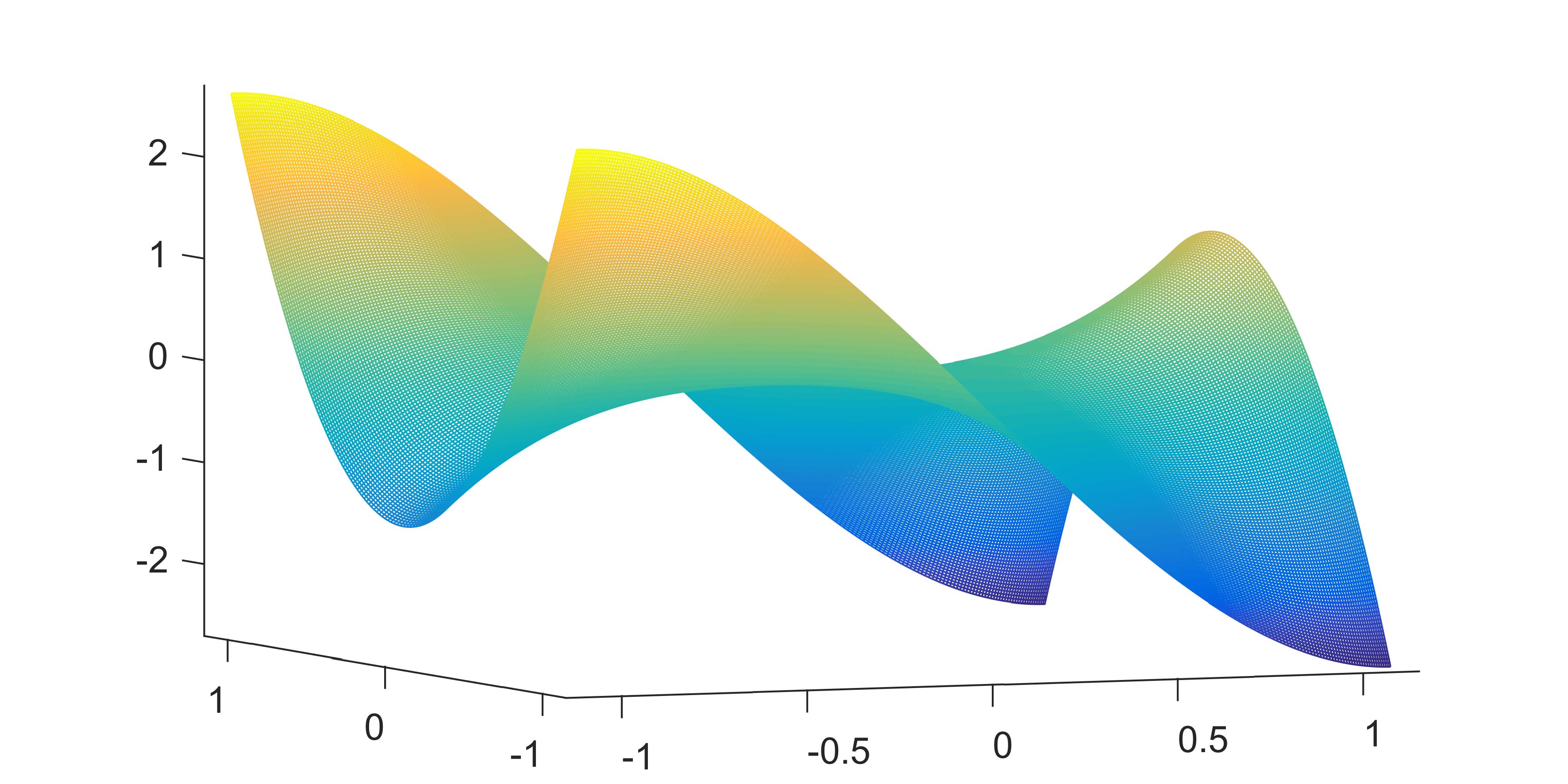}}
    \subfigure{
	     \includegraphics[width=0.48\textwidth]{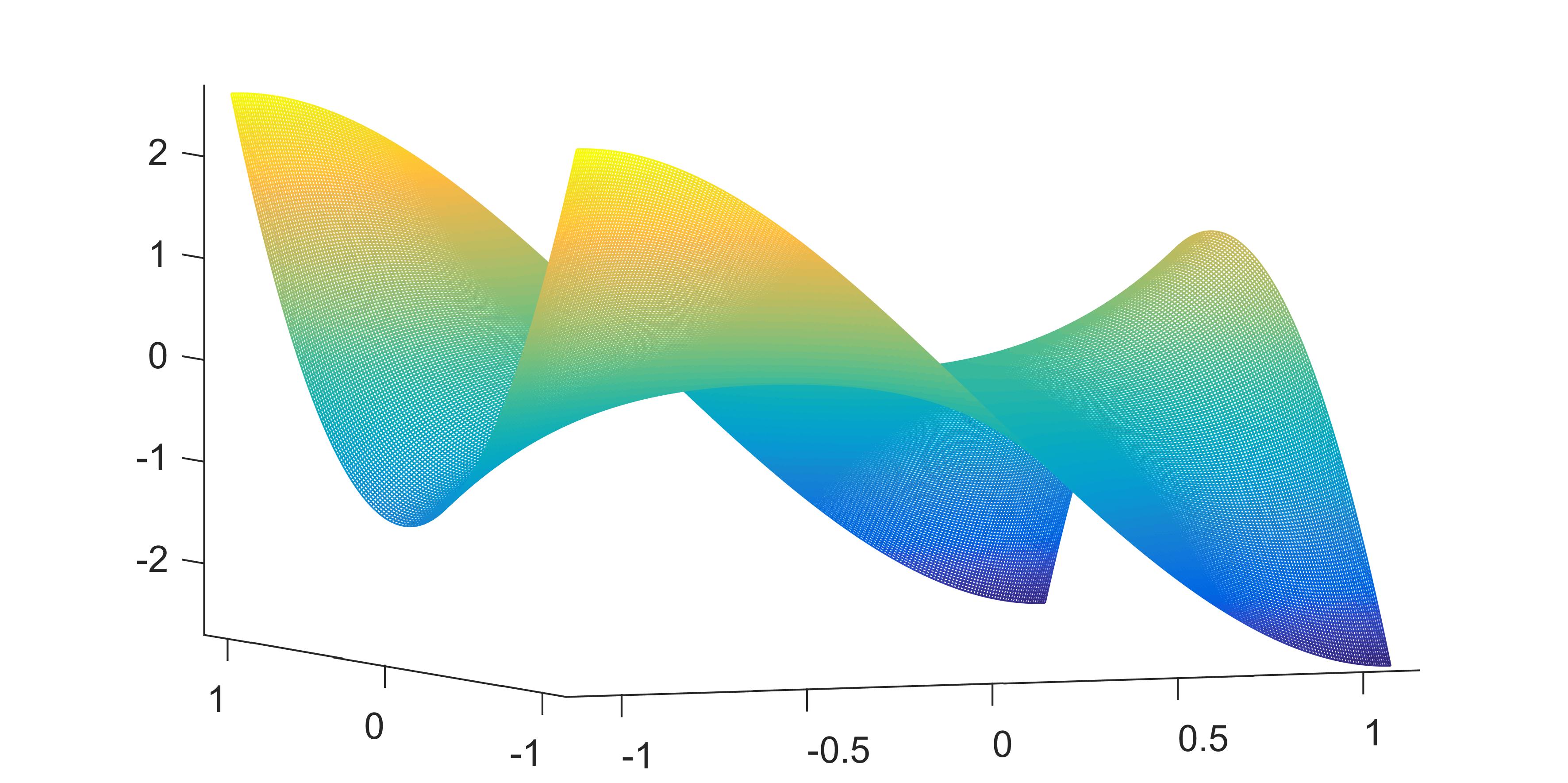}}
\caption{The cubic B-spline LSPIA (left) and ARBK($50$, $50$) (right) fitting surfaces using $50 \times 50$ control coefficients to fit data points $1$ in Example \ref{ARBK:Ex:CAGD-AXB=C}.}
\label{ARBK:fig:Ex3-21}
\end{figure}

\begin{figure}[!htb]
\centering
    \subfigure{
	     \includegraphics[width=0.48\textwidth]{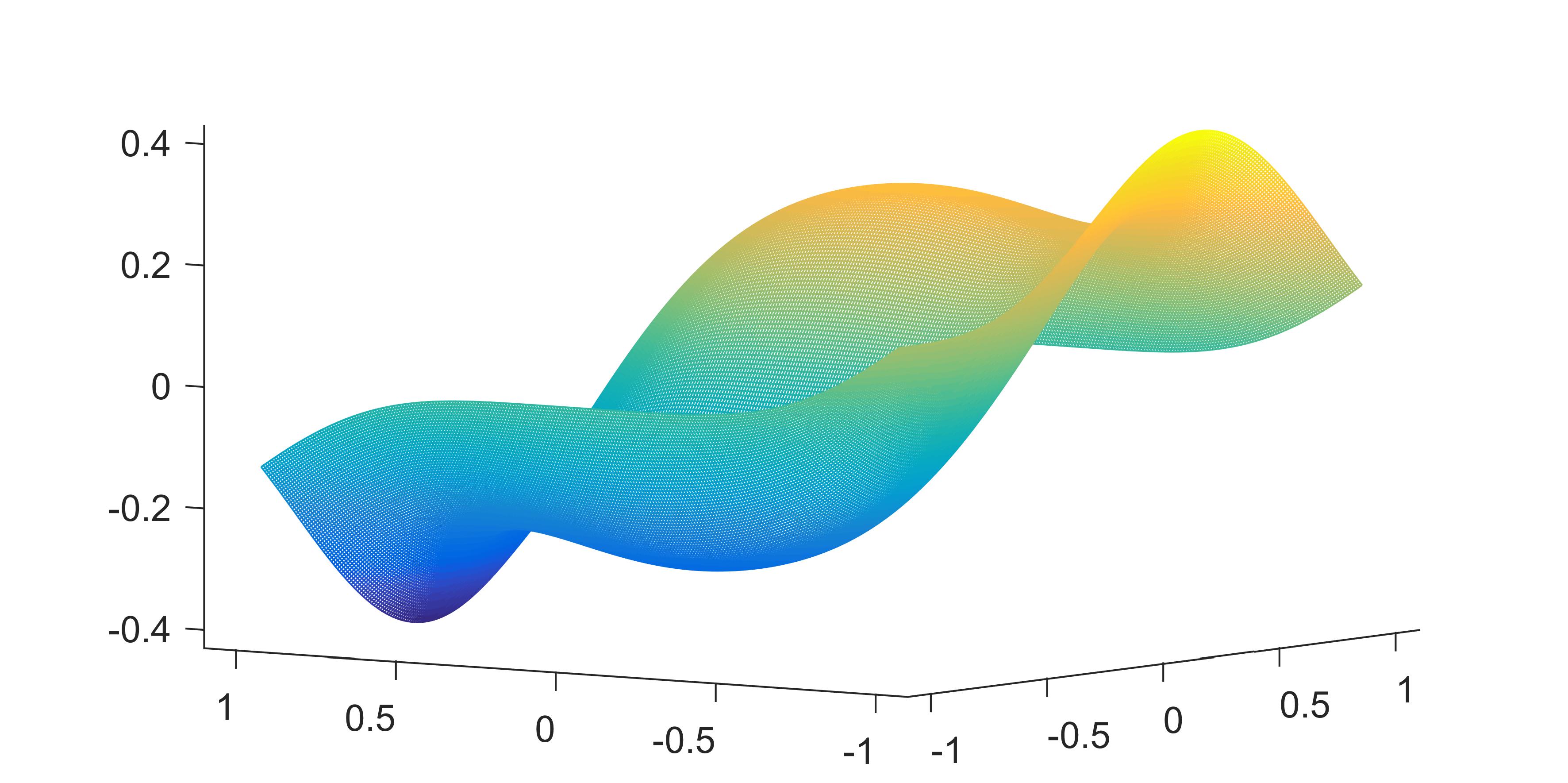}}
    \subfigure{
	     \includegraphics[width=0.48\textwidth]{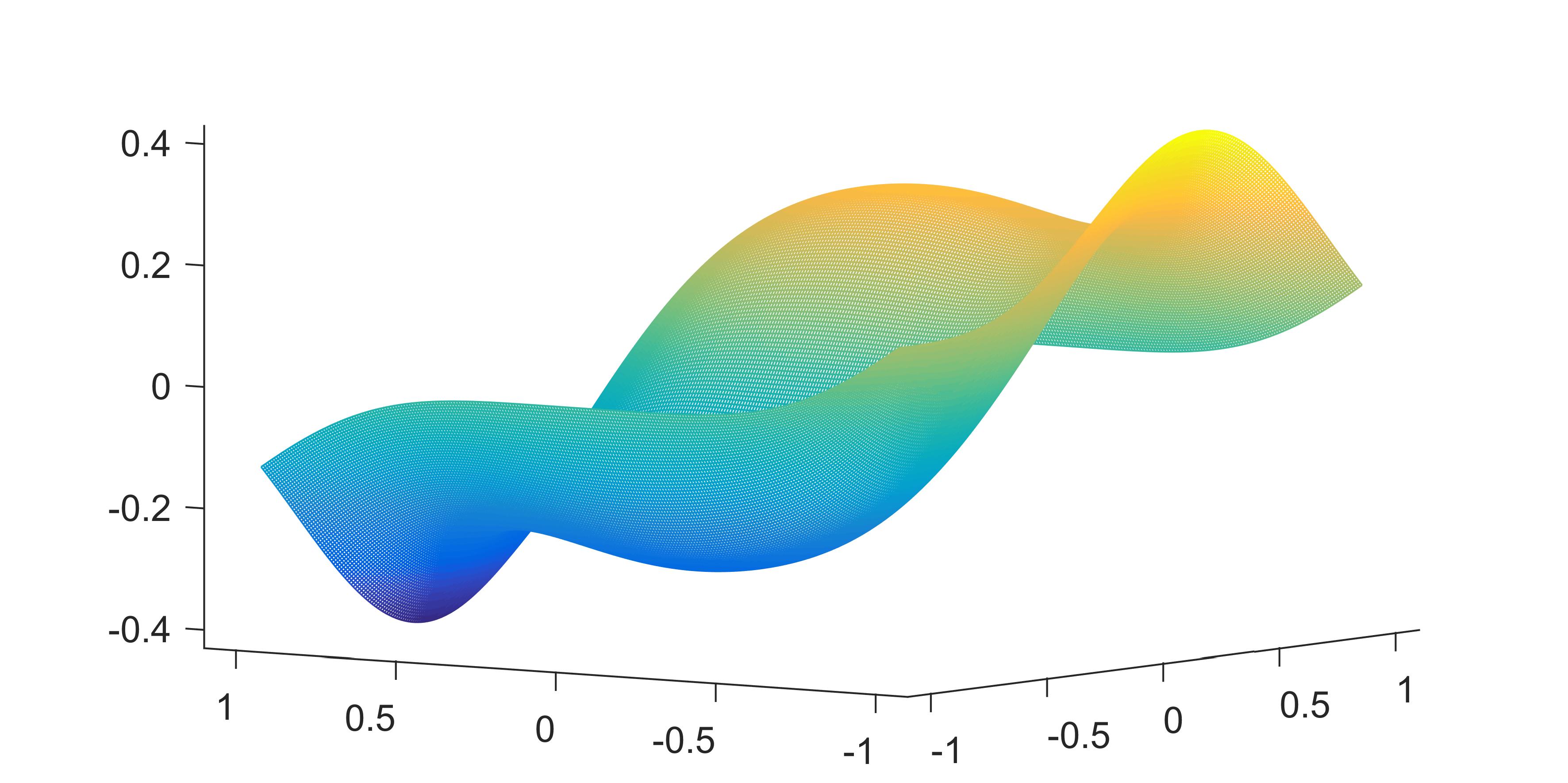}}
\caption{The cubic B-spline LSPIA (left) and ARBK($50$, $50$) (right) fitting surfaces using $50 \times 50$ control coefficients to fit data points $2$  in  Example \ref{ARBK:Ex:CAGD-AXB=C}.}
\label{ARBK:fig:Ex3-22}
\end{figure}
\section{Conclusions}\label{AXB=F:Conclusion}
The classical block Kaczmarz method has been developed to solve the over-determined linear systems. It operates by projecting the current point onto the solution space of a subset of constraints at each iteration. In this work, we replace the random single row selection $A_{i,:}$ and column selection $B_{:,j}$ in the CME-RK method \cite{23XBL} with the random  multiple row selection $A_{U_i,:}$ and column selection $B_{:,V_j}$, respectively, and present an alternating randomized block Kaczmarz method. The key innovation of this method is alternately utilizing the randomized block iteration scheme. The convergence analysis demonstrates that the resulting algorithm achieves a linear convergence rate bounded by an explicit expression. Numerical results illustrate that the proposed method is a competing randomized algorithm for solving the general consistent matrix equation.

Finally, we point out that finding the optimal value of block-sizes in our method is a technical and skillful issue. This parameter is determined by various factors such as the concrete structure and property of the coefficient matrices. In this work, we have not discussed how to determine these parameters effectively for practical examples. This topic is of real value and theoretical importance. We will investigate this in detail in the future.

\vskip 2ex
\noindent{\large\bf Declarations}
\vskip 1.15ex

\noindent{\bf Ethics approval.} Not applicable.
\vskip 1.15ex

\noindent{\bf Availability of supporting data.}  Data sharing not applicable to this article as no datasets were generated or analyzed during the current study.
\vskip 1.15ex

\noindent{\bf Competing interests.} This study does not have any conflicts to disclose.
\vskip 1.15ex

\noindent{\bf Funding.} The authors express their appreciation for supports provided by the National Natural Science Foundation of China under grant  12201651 and the Fundamental Research Funds for the Central Universities, South-Central University for Nationalities under grant CZQ23004. The third author was also supported by the Natural Science Foundation of Shanxi Province under grant  20210302123480 and the Research Project Supported by
Shanxi Scholarship Council of China (2023-117). These supports are gratefully acknowledged.
\vskip 1.15ex

\noindent{\bf Author's contributions.} All authors contributed equally to the manuscript and read and approved the final manuscript.
\vskip 1.15ex


\end{document}